\documentclass[final,1p,times]{elsarticle}

\usepackage{amssymb}
\usepackage{amsmath}
 \usepackage{amsthm}
\usepackage[dvipsnames]{xcolor}
\usepackage{amsfonts}
\usepackage{dsfont}
\usepackage{hyperref}
\allowdisplaybreaks

\newcommand{\lb}{\left(}
\newcommand{\rb}{\right)}

\newcommand{\bfA}{{\mathbf A}}
\newcommand{\bfT}{{\mathbf T}}

\newcommand{\bfR}{{\mathbf R}}
\newcommand{\SR}{{\mathcal R}}
\newcommand{\bfD}{{\mathbf D}}

\newcommand{\bfM}{{\mathbf M}}

\newcommand{\bfI}{{\mathbf I}}
\newcommand{\bfX}{{\mathbf X}}
\newcommand{\bfY}{{\mathbf Y}}

\newcommand{\bfu}{{\mathbf u}}

\newcommand{\beao}{\begin{eqnarray*}}
\newcommand{\eeao}{\end{eqnarray*}}
\newcommand{\beam}{\begin{eqnarray}}
\newcommand{\eeam}{\end{eqnarray}}
\newcommand{\barr}{\begin{array}}
\newcommand{\earr}{\end{array}}

\newcommand{\bco}{\begin{corrolary}}
\newcommand{\eco}{\end{corrolary}}

\newcommand{\E}{\mathbb{E}}
\renewcommand{\P}{\mathbb{P}}
\newcommand{\1}{\mathds{1}}
\newcommand{\R}{\mathbb{R}}
\newcommand{\N}{\mathbb{N}}
\newcommand{\C}{\mathbb{C}}
\newcommand{\bfC}{{\mathbf C}}

\newcommand{\bfB}{{\mathbf B}}

\newcommand{\Var}{\operatorname{Var}}

\newcommand{\Levy}{L\'{e}vy }

\newcommand{\x}{{\mathbf x}}
\newcommand{\y}{{\mathbf y}}

\newcommand{\X}{{\mathbf X}}
\newcommand{\Y}{{\mathbf Y}}

\newcommand{\bfZ}{{\mathbf Z}}
\newcommand{\z}{{\mathbf Z}}

\newcommand{\dint}{\,\mathrm{d}}

\newcommand{\twonorm}[1]{\|#1\|}

\newcommand{\vep}{\varepsilon}
\newcommand{\nto}{n \to \infty}
\newcommand{\pto}{p \to \infty}

\newcommand{\lhs}{left-hand side}

\newcommand{\tr}{\operatorname{tr}}

\newcommand{\sign}{\operatorname{sign}}

\newcommand{\diag}{\operatorname{diag}}
\newcommand{\offdiag}{\operatorname{offdiag}}

\newcommand{\MP}{Mar\v cenko--Pastur }

\renewcommand{\Im}{\operatorname{Im}}

\newcommand{\gn}{\gamma_p}

\newcommand{\as}{{\rm a.s.}}

\newcommand{\cip}{\stackrel{\P}{\rightarrow}}

\newcommand{\cas}{\stackrel{\rm a.s.}{\longrightarrow}}

\newtheorem{lemma}{Lemma}[section]

\newtheorem{theorem}[lemma]{Theorem}

\newtheorem{assumption}[lemma]{Assumption}

\newtheorem{example}[lemma]{Example}

\newtheorem{remark}[lemma]{Remark}

\journal{Stochastic Processes and their Applications}

\begin{document}

\begin{frontmatter}



\title{Ties, Tails and Spectra: On Rank-Based Dependency Measures in High Dimensions}


\author[nina]{Nina Dörnemann} 

\affiliation[nina]{organization={Department of Mathematics, Aarhus University},
            addressline={Ny Munkegade 118}, 
            city={Aarhus C},
            postcode={8000}, 
            country={Denmark}}

            \author[michael]{Michael Fleermann}
\affiliation[michael]{organization={School of Life Science, Engineering \& Design, 
Saxion University of Applied Sciences},
             country={The Netherlands}
            }

\author[johannes]{Johannes Heiny}
\affiliation[johannes]{organization={Department of Mathematics,
KTH Royal Institute of Technology},
            addressline={Lindstedtsvägen 25}, 
            city={Stockholm},
            postcode={11428}, 
            country={Sweden}}

\begin{abstract}
This work is concerned with the limiting spectral distribution of rank-based dependency measures in high dimensions. We provide distribution-free results for multivariate empirical versions of Kendall's $\tau$ and Spearman's $\rho$ in a setting where the dimension $p$ grows at most proportionally to the sample size $n$. Throughout, “distribution-free” is used in the sense that the limiting distribution is pivotal with respect to the marginal distributions of the data. Although rank-based measures are known to be well suited for discrete data, previous works in the field focused mostly on the continuous case. We close this gap by imposing mild assumptions and allowing for general types of distributions.
Interestingly, our analysis reveals that a non-trivial adjustment of classical Kendall's $\tau$ is needed to obtain a pivotal limiting distribution in the presence of tied data.  
The proof for Spearman's $\rho$ is facilitated by a result regarding the limiting eigenvalue distribution of a general class of random matrices with rows on the Euclidean unit sphere, which is of independent interest. For instance, this finding can be used to derive the limiting spectral distribution of sample correlation matrices, which, in contrast to most existing works, accommodates heavy-tailed data.
\end{abstract}




\end{frontmatter}

\section{Introduction}
Statistical inference for dependency measures is a fundamental yet challenging problem, with a variety of applications across many fields. 
In traditional multivariate analysis, the covariance matrix and its normalized counterpart, the correlation matrix, have been widely employed for this purpose. However, the usefulness of such parametric measures often relies on restrictive assumptions, particularly concerning the moments or even the normality of data. For instance, these measures cannot be used to test for independence among uncorrelated yet dependent random variables, and they cannot even be defined when the variance is infinite. Rank-based measures have become a popular tool to overcome such challenges, as they assess the relative positions of data points rather than their actual values. Therefore, rank-based measures are robust against outliers, which is advantageous for heavy-tailed data. Additionally, they can capture nonlinear relationships between variables. Two among the most well-known measures, namely 
Kendall's $\tau$ \cite{Kendall1938} and Spearman's $\rho$ \cite{spearman1987proof} and their natural multivariate extensions are discussed in this work.

Spurred by the wide availability of data sets involving many parameters across diverse disciplines such as biostatistics, wireless communications, and finance \cite{Fan2006, Johnstone2006}, rank-based statistics such as multivariate Kendall's $\tau$ and Spearman's $\rho$ are frequently studied in high-dimensional settings. 
For instance, these non-parametric dependency measures are used for independence testing in high dimensions (see, e.g., \cite{shi2022universally, bastian2022independence, han2017distribution}).
The primary focus of this work lies in investigating the asymptotic spectral distributions of multivariate Spearman's $\rho$, Kendall's $\tau$ and a general class of random matrices including the sample correlation matrix. Since the pioneering works \cite{marchenko:pastur:1967, wigner:1955, wigner:1957}, 
the limiting eigenvalue distribution of various types of random matrices received a lot of interest in the literature: For example, \cite{silverstein1995analysis, fleermann:heiny:2023, fleermann:heiny:2020} on sample covariance matrices, \cite{bryc2006spectral, catalano:fleermann:2024, fleermann:kirsch:2021} on Hankel, Markov, Toeplitz and band matrices, \cite{bose2009limiting} on circulant type matrices, \cite{wang:yao:2016} on auto-covariance matrices, \cite{li:wang:yao:2022} on spatial-sign covariance matrices, \cite{li_et_al_2023} on distance covariance matrices, just to name a few. Such results can be seen as a first meaningful step towards understanding the second-order asymptotics of linear spectral statistics, which facilitates a variety of statistical problems.
Recently, the spectral analysis of non-parametric dependency measures has attracted a lot of interest in random matrix theory. The Tracy--Widom law for the leading eigenvalues of empirical Kendall's $\tau$ and Spearman's $\rho$ matrix are established in \cite{bao2019tracy} and \cite{bao2019tracy_sperman}, respectively. The limiting bulk distribution of Spearman's $\rho$ is investigated in \cite{bai:zhou:2008, wu2022limiting}. For Kendall's $\tau$, we refer to \cite{bandeira:lodhia:rigollet:2017, wang2021eigenvalues, bousseyroux:espana:smerlak:2026}. The linear spectral statistics of Spearman's and Kendall's rank correlation matrix are studied in \cite{bao:lin:pan:zhou:2015} and \cite{li2021central}, respectively. 

Despite the merits of rank-based measures for tied 
data, the question of their spectral properties is largely unresolved. In fact, the aforementioned works concentrate on the case of continuous 
data, and we are not aware of any results regarding the limiting eigenvalue distribution in the setting considered here. Moreover, most of the works in random matrix theory are concerned with the high-dimensional case, where the dimension is of the same magnitude as the sample size, while the moderately high-dimensional regime plays only a minor role. However, in many fields, the wider range of possible growth rates arising in the moderately high-dimensional setting is desirable (see, e.g., \cite{elkaroui:2003}). In this work, we provide a unifying approach for these two asymptotic regimes, with results applicable both for tied and heavy-tailed data. 
Our main contributions can be summarized as follows.

\begin{enumerate}
\item We provide the limiting eigenvalue distributions of Kendall's $\tau$, Spearman's $\rho$ and a general class of random matrices including sample correlation matrices. The paper's main contribution is not merely another limiting spectral distribution result, but to identify which rank-based matrices remain spectrally universal in the presence of ties and heavy tails, and which objects have to be renormalized or adapted.
\item Two standard asymptotic frameworks in random matrix theory are considered, where the dimension $p$ and the sample size $n$ diverge. We provide results both for the case where $p$ is negligible compared to $n$ and where $p$ is asymptotically proportional to $n.$ 
\item We impose mild assumptions on the data-generating distribution, allowing for both heavy-tailed and tied data. 
 \item We show that the Spearman's $\rho$ matrix has a universal limiting spectral distribution if a suitable self-normalization is used. 
\item The classical multivariate Kendall's $\tau$ matrix is not spectrally distribution-free under ties.
We adapt the classical multivariate version of Kendall's $\tau$ and find an appropriate scaling to restore universality. 
\end{enumerate}
The remainder of this paper is structured as follows.  In Section~\ref{sec:pre}, we provide some preliminaries and introduce necessary notations. 
Section~\ref{sec:kendalls:spearman} is devoted to the investigation of Kendall's $\tau$ and Spearman's $\rho$ matrices, whereas Section \ref{sec_unit_sphere} contains the results on random matrices whose rows live on the Euclidean unit sphere. We will also discuss the application to sample correlation matrices. All longer proofs are deferred to Sections \ref{sec_proof_spearman_unit_sphere} and \ref{sec_proof_kendall}.

\section{Eigenvalue distributions of several dependency measures} 
\subsection{Preliminaries} \label{sec:pre} 
In this paper, the sample size $n$ is a function of the dimension $p$ and the dimension increases at most proportionally to the sample size. To be precise, we assume 
\begin{equation*}
n=n_p \to \infty \quad \text{ and } \quad \frac{p}{n_p}\to \gamma\in [0,\infty)\,,\quad \text{ as } \pto\,. 
\end{equation*}
The constant $\gamma$ controls the growth of the dimension relative to the sample size. Our results cover the cases $\gamma=0$ and $\gamma>0$. Throughout this work, let 
\begin{equation}\label{eq:data}
\bfX:= \X_p := (X_{ij})_{i=1,\ldots,p;\, j=1,\ldots,n} \in\R^{p\times n}
\end{equation}
be a random matrix whose rows are independent with i.i.d.\ components while the distributions of the rows might differ.
\medskip

\textbf{Spectral distributions.}
Let $\bfA_p$ be a $p\times p$ (possibly random) matrix with real eigenvalues $\lambda_1(\bfA_p) \geq \ldots \geq \lambda_p (\bfA_p).$ The uniform distribution on its eigenvalues is called {\em empirical spectral distribution} of $\bfA_p$ and is defined by 
\begin{equation*}
F_{\bfA_p}(x)= \frac{1}{p}\; \sum_{i=1}^p \1( \lambda_i(\bfA_p)\le x), \qquad x\in  \R\,.
\end{equation*}
A major problem in random matrix theory is to find the weak limit of $(F_{\bfA_p})$ for suitable sequences $(\bfA_p)$; see for example \cite{bai:silverstein:2010,yao:zheng:bai:2015} for more details. To be precise, we understand the weak convergence of a sequence of probability distributions $(F_{\bfA_p})$ 
to a probability distribution $F$ as $\lim_{\pto} F_{\bfA_p}(x)=F(x)$ almost surely for all continuity points of $F$. This weak limit is called \textit{limiting spectral distribution} of $(\mathbf{A}_p).$ Two types of limiting distributions are ubiquitous in random matrix theory, which we will revisit and explain their appearance in the special case of the sample covariance matrix. 
\medskip

\textbf{\MP distribution.}
The \textit{\MP distribution} $F_{\gamma}$ appears as the limiting spectral distribution of the sample covariance matrix in the regime $\gamma >0 $ (see \cite{bai:silverstein:2010, bai:zhou:2008}).
For $\gamma \in (0,1]$,  $F_\gamma$  has density
\begin{eqnarray}\label{eq:MPch1}
f_\gamma(x) :=
\left\{\begin{array}{ll}
\frac{1}{2\pi x\gamma} \sqrt{(b-x)(x-a)} \,, & \mbox{if } a\le x \le b, \\
 0 \,, & \mbox{otherwise,}
\end{array}\right.
\end{eqnarray}\noindent
where $a=(1-\sqrt{\gamma})^2$ and $b=(1+\sqrt{\gamma})^2$. For $\gamma>1$, the \MP law has an additional point mass $1-1/\gamma$ at $0$.

In the case $\gamma=0$, the limiting spectral distribution of the sample covariance matrix is the Dirac measure at
$1$. After an appropriate transformation of its spectrum, one can obtain a non-degenerate limiting spectral distribution.
\medskip

\textbf{Semicircle law.}
Under certain conditions, the empirical spectral distribution of the rescaled sample covariance matrices
 converges to the so-called \textit{semicircle law} $G$ (see \cite{wang:paul:2014}).
The density of $G$ is given by
\begin{equation}\label{eq:semicircle}
g(x):= \tfrac{1}{2\pi} \sqrt{4-x^2}\, \1_{[-2,2]}(x) \,, \qquad x\in \R\,.
\end{equation}
\begin{remark}{\em 
The semicircle law arises from the \MP\ distribution $F_{\gamma}$ in the following way. Consider a random variable 
$Z_{\gamma}$ with distribution $F_{\gamma}$ for $\gamma\in (0,1]$. The density of $\gamma^{-1/2} (Z_{\gamma} -1)$ is given by
\begin{equation*}
f_{\gamma}(\gamma^{1/2} x+1)\gamma^{1/2}=\frac{1}{2\pi (\gamma^{1/2} x +1)}\sqrt{ (\gamma^{1/2}+2-x)(-\gamma^{1/2}+2+x)} \, \1_{[-2+\gamma^{1/2},2+\gamma^{1/2}]}(x)\,,\quad x\in \R\,.
\end{equation*}
Taking the limit $\gamma \searrow 0$ we see that $f_{\gamma}(\gamma^{1/2} x+1)\gamma^{1/2}$ converges to $g(x)$ in \eqref{eq:semicircle}.
}\end{remark}
  Finally, we need the following definitions. If $\bfC$ is a square matrix, $\diag(\bfC)$ denotes the diagonal matrix which has the same diagonal elements as $\bfC$. Sometimes we will simply refer to $\diag(\bfC)$ as the diagonal of $\bfC$. Analogously, we define $\offdiag(\bfC)=\bfC-\diag(\bfC)$. The matrix $\bfI$ denotes the identity matrix of appropriate dimension.

\subsection{Two rank-based measures} \label{sec:kendalls:spearman}
 In this section, we discuss two rank-based dependency measures, namely multivariate extensions of Kendall's $\tau$ \cite{Kendall1938} and Spearman's $\rho$ \cite{spearman1987proof}. 
  Unlike certain other measures, Kendall's $\tau$ and Spearman's $\rho$ do not require the data to be light-tailed or even continuous, making them useful for analyzing ordinal data. Both measures are robust to the presence of outliers and ties. 

We proceed by writing $Q_{ij}$ for the (fractional) rank of $X_{ij}$ in ascending order among $X_{i1},\ldots, X_{in}$, that is, we set 
\begin{equation*}
 Q_{ij}:=\sum_{t=1}^n \1(X_{it}\le X_{ij})   - \frac{1}{2} \sum_{t=1; \, t\neq j}^n \1(X_{it}=X_{ij})\,.  
\end{equation*}
This means, for example, that the largest of the $X_{ij}$'s has the largest (fractional) rank.  It is worth mentioning that $\sum_{j=1}^n Q_{ij}=n(n+1)/2$ whether or not there are ties among the $X_{ij}$'s.

In order to introduce the natural multivariate extension of Spearman's $\rho$, we define the $p\times n$ matrix $\z$ with entries
\begin{equation}\label{eq:z}
Z_{ij}:=  \frac{Q_{ij}-\frac{n+1}{2}}{\sqrt{\sum_{t=1}^n \big(Q_{it}-\frac{n+1}{2}\big)^2 }}\,, \qquad i=1,\ldots,p; j=1,\ldots,n.
\end{equation}
Then, Spearman's rank correlation matrix is defined by
\begin{equation}\label{eq:defspear}
\SR := \z\z',
\end{equation}
which is nothing other than Pearson's correlation matrix of the ranks $Q_{ij}$, $1 \leq i \leq p, 1 \leq j \leq n.$ 
Note that the denominator in \eqref{eq:z} ensures that the rows of $\z$ lie on the unit sphere in $\R^n$. In contrast to other works \cite{bao2019tracy, li2021central}, we do not rely on continuous distributions and only require a mild assumption, which we refer to as asymptotic non-degeneracy.
This assumption is necessary to avoid the case that $Q_{ij}=\frac{n+1}{2}$ for all $j$. 
\begin{assumption} \label{ass_asympt_non_degen}
The scheme $(X_{i1})_{1\leq i \leq p, p\in\N}$ is \textit{asymptotically uniformly non-degenerate}, that is, we assume that for all (real-valued) triangular schemes $(y_p^{(i)})_{1 \leq i \leq p, p\in\N}$ 
\begin{align*} 
    \limsup_{p\to\infty} \max_{1\leq i \leq p} \P \lb X_{i1} = y_p^{(i)} \rb < 1. 
\end{align*}
\end{assumption}

In the following theorem, we provide a distribution-free result on the eigenvalue distribution of $\SR$. Its proof can be found in Section \ref{sec_proof_rho}. 
\begin{theorem}[Spearman's Rho]
\label{thm:spearman}
Suppose that Assumption \ref{ass_asympt_non_degen} is satisfied. 
\begin{enumerate}
\item[(1)] \label{spearman}
If $p/n\to 0$, then, as $p\to\infty$, the empirical spectral distribution of $\sqrt{n/p} \, (\SR-\bfI)$ converges almost surely to the semicircle law $G$.
\item[(2)] \label{spearman'} 
If $p/n\to \gamma>0$, then, as $p\to\infty$, the empirical spectral distribution of $\SR$ converges almost surely to the \MP law $F_{\gamma}$. 
\end{enumerate}
\end{theorem}

\begin{remark}
{\em 
One of the strengths of our findings is that they accommodate distributions that allow for ties, which is crucial for the application of Spearman's rho to discrete data.  Nevertheless, it is important to compare our findings with the existing results for the continuous case. If the distribution of $X_{i1}$ has no point mass for $1 \leq i \leq p$, we observe that $\{ Q_{i1}, \ldots, Q_{in} \} = \{ 1, \ldots, n\}$ almost surely, and from \eqref{eq:z} we obtain the representation 
\begin{equation*}
Z_{ij}= \sqrt{ \frac{12}{n(n^2-1)}} \bigg(Q_{ij}-\frac{n+1}{2} \bigg)\,, \qquad i=1,\ldots,p; j=1,\ldots,n,
\end{equation*}
which aligns with the definition of \cite{bao:lin:pan:zhou:2015, bai:zhou:2008}. Thus, in the case of continuous distributions and $\gamma > 0$, we recover the result of \cite[Theorem 2.2]{bai:zhou:2008}.
} 
\end{remark}
Next, we turn to the investigation of the eigenvalue distribution of Kendall's $\tau$, which is based on the number of concordant and discordant pairs of observations in the sample.
For this purpose, let $(A_1,B_1), \ldots, (A_n,B_n)$ be independent observations of a pair $(A,B)$ of real-valued random variables. Then the (empirical) Kendall's $\tau$ is defined as 
\begin{equation*}
\tau=\tfrac{2}{n(n-1)} \sum_{1\le s<t\le n} \sign(A_s-A_t) \sign(B_s-B_t) \,,
\end{equation*}
where $\sign(0)=0$. Clearly, Kendall's $\tau$ takes values in $[-1,1]$.
An equivalent definition is given by
\begin{equation*}
\tau=\tfrac{2}{n(n-1)}  (\#\{ \text{concordant pairs} \} -\#\{ \text{discordant pairs} \})\,.
\end{equation*}
Moreover, Kendall's $\tau$ is a U-statistic and, by standard asymptotic theory, it can be shown that it is asymptotically normal 
\cite{vandervaart:1998}. 
Now, let us turn to higher dimensional observations, such as the columns of the data matrix $\X$. 
Then, the empirical Kendall's $\tau$ of $\bfX$ is defined as the $p\times p$ matrix $\boldsymbol{\tau}=(\tau_{ij})$ with entries
\begin{equation*}
\tau_{ij}:=\tfrac{2}{n(n-1)}  \sum_{1\le s<t\le n} \sign(X_{is}-X_{it}) \sign(X_{js}-X_{jt})\,, \qquad 1\le i,j\le p\,.
\end{equation*}
In this sense, the matrix $\boldsymbol{\tau}$ can be seen as a natural multivariate extension of the coefficient $\tau$. In the context of discrete data, Kendall's $\tau_B$ is more suitable adjusting for ties (for a definition, see \cite{agresti2010analysis}). It aligns with the usual Kendall's $\tau$ in the case that the data follow a continuous distribution.
However, it turns out that neither $\boldsymbol\tau$ nor the multivariate version of $\tau_B$, i.e., the matrix of pairwise $\tau_B$ coefficients, admit a distribution-free LSD (in the presence of ties), and the required scaling is not readily apparent (for more details, see Remark \ref{remark_scaling_D}). For discrete distributions, the $\tau_{ij}$'s will in general have different variances and $\E[\tau_{ii}]=\P(X_{i1}\neq X_{i2})< 1$. This constitutes a major difference to the case of continuous distributions.  Additional insight can be gained from the very recent works \cite{benaych:espana:2026, shevade:bhattacharjee:2026}. Hence, it becomes imperative to find an alternative version of Kendall's $\tau$, which has an eigenvalue distribution in the limit independent of the data distribution. To overcome this issue, we propose the following version of $\boldsymbol{\tau}$, which only depends on the ranks:
\begin{equation*}
\bfT:= \bfD^{-1/2}  \lb \tfrac{2}{n(n-1)} \sum_{1\le s<t\le n} \offdiag \Big(\sign(\mathbf{q}_s-\mathbf{q}_t) (\sign(\mathbf{q}_s-\mathbf{q}_t))'\Big)  \rb \bfD^{-1/2} ,
\end{equation*} 
where $\sign$ of a vector is taken coordinate-wise, $\mathbf{q}_s = (Q_{1s}, \ldots, Q_{ps})'$, and $\bfD$ is a $p\times p$ diagonal matrix with diagonal entries
\begin{align} \label{eq_def_D}
    D_{ii} &:=  \frac{12}{n^3} \sum_{s=1}^n \lb Q_{is} - \frac{n+1}{2} \rb ^2, \qquad 1 \leq i \leq p.
\end{align} 
It is worth noting that $\bfT= \bfD^{-1/2} \offdiag(\boldsymbol{\tau})\,\bfD^{-1/2}$. The definition of $\bfD$ is discussed in Remark \ref{remark_scaling_D}.
The following theorem provides the limiting spectral distribution of $\bfT$, and its proof is deferred to Section \ref{sec_proof_kendall}. 
\begin{theorem}[Kendall's $\tau$]
\label{thm:kendall}
Suppose that Assumption \ref{ass_asympt_non_degen} is satisfied. 
\begin{enumerate}
\item[(1)] 
If $p/n\to 0$, then, as $p\to\infty$, the empirical spectral distribution of 
$ \sqrt{n/p} \,  \bfT$ 
converges in probability to the distribution of $\frac{2}{3}\zeta$, where $\zeta$ follows the semicircle law. 
\item[(2)] 
If $p/n\to \gamma>0$, then, as $p\to\infty$, the empirical  spectral distribution of $\bfT$
converges in probability to the distribution of $\tfrac{2}{3} (\eta-1)$, where $\eta$ follows the \MP law with parameter $\gamma$.
\end{enumerate}
\end{theorem}

\begin{figure}[htbp]
  \centering
  \includegraphics[width=\linewidth]{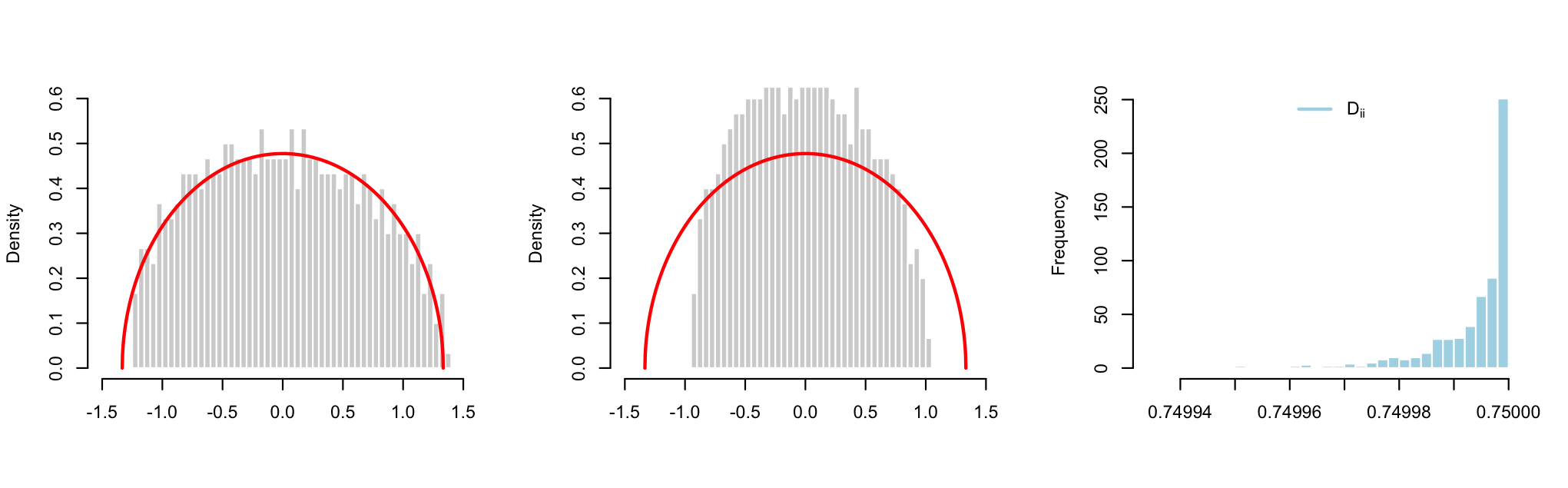}
  \caption{
 Normalized histograms of the simulated eigenvalues of $\sqrt{n/p} \cdot \bfT$ (left panel) and $\sqrt{n/p}\cdot \offdiag(\boldsymbol{\tau})$ (middle panel), and histogram of the diagonal entries of $\bfD$ (right panel), where $(p,n)=(600,120000)$ and $\bfX$ has i.i.d.\ $Ber(0.5)$-entries. The red curve indicates the density of the scaled semicircle law $\frac{2}{3} \zeta$, which is the limiting spectral distribution of $\sqrt{n/p} \cdot \bfT$ for $p/n\to 0$ (part $(1)$ of Theorem~\ref{thm:kendall}). 
 }
\label{fig:kendall_scl}
\end{figure}

The convergence in probability can be extended to almost sure convergence, but for technical reasons it will not be pursued in this work. In Theorem \ref{thm:kendall}, we have proposed an adapted version of Kendall's tau for the high-dimensional regime $p/n\to\gamma>0$ and the moderately high-dimensional regime $p/n\to 0$. For future research, it would be interesting to study the ultra-high-dimensional regime $p/n\to\infty$ using the tools developed in this work. A first step in this direction is the work of \cite{bousseyroux:espana:smerlak:2026}, who consider continuous data in the regime $p=n^2$.

In the following remark, we consider the continuous case and draw a connection to an existing result in the case $\gamma >0.$
\begin{remark} \label{rem_con_kendall}
{\em  
 Again, we compare our findings to the case where the rows of $\bfX$ are generated by continuous distributions. In this case, we have almost surely
 \begin{align*}
     D_{ii} 
     = \frac{12}{n^3} \sum_{s=1}^n \lb Q_{is} - \frac{n+1}{2} \rb ^2
     =\frac{12}{n^3} \lb \sum_{s=1}^n s^2 - 2 \frac{n+1}{2} \sum_{s=1}^n s + n \lb \frac{n+1}{2} \rb^2 \rb  
     = 1 + o(1), \qquad 1 \leq i \leq p,
 \end{align*}
 and 
 \begin{equation*}
     \tfrac{2}{n(n-1)} \sum_{1\le s<t\le n} \diag \Big(\sign(\mathbf{q}_s-\mathbf{q}_t) (\sign(\mathbf{q}_s-\mathbf{q}_t))'\Big)=\bfI\,.
 \end{equation*}
 This implies that $\bfT=  (1+o(1))\boldsymbol{\tau} -\bfI$ and hence the LSDs of $\bfT + \bfI$ and $\boldsymbol{\tau}$ coincide. 
    That is, in the case $p/n\to \gamma \in (0,\infty)$ and for continuous $X_{11}, \ldots, X_{p1}$, we recover the result of Bandeira et al. \cite{bandeira:lodhia:rigollet:2017}, who proved that the empirical spectral distribution of the matrix
    \begin{align*}
    & \tfrac{2}{n(n-1)} \sum_{1\le s<t\le n} \sign(\x_s-\x_t) (\sign(\x_s-\x_t))' \\ 
       = & ~ \tfrac{2}{n(n-1)} \sum_{1\le s<t\le n} \sign(\mathbf{q}_s-\mathbf{q}_t) (\sign(\mathbf{q}_s-\mathbf{q}_t))'   =
        \bfD^{1/2} \bfT \bfD^{1/2} + \bfI
    \end{align*}
 converges weakly in probability to $\tfrac{2}{3} \eta + \tfrac{1}{3}$. 
 Moreover, if the prefactor $\frac{12}{n^3}$ in the definition of $D_{ii}$ were replaced by $\frac{12}{n(n^2-1)}$, the $o(1)$ term would disappear. Asymptotically this does not make any difference and all statements of Theorem~\ref{thm:kendall} would remain valid.
}
\end{remark}

\begin{figure}[htb!]
  \centering
  \includegraphics[width=\linewidth]{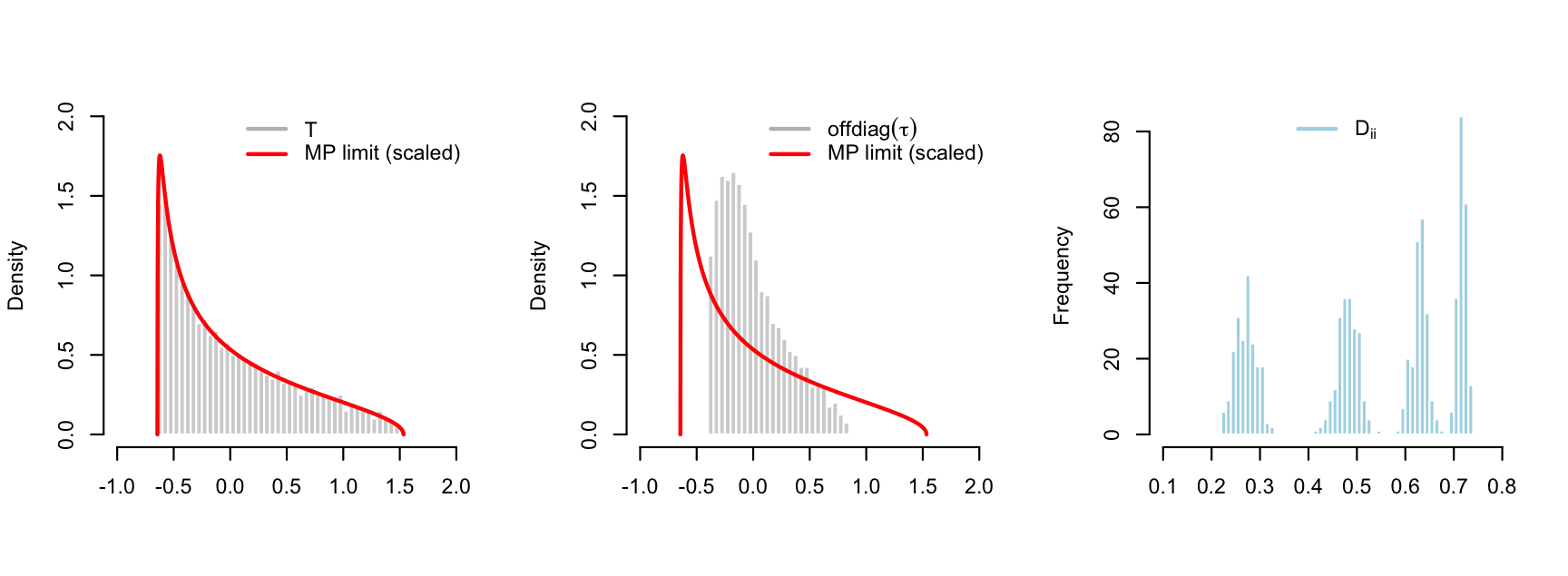}
  \caption{Normalized histograms of the simulated eigenvalues of $\bfT$ (left panel), $\offdiag(\boldsymbol{\tau})$ (middle panel) and of the diagonal entries of $\bfD$ (right panel), where $(p,n)=(800,1200)$ and every fourth row of $\bfX$ has entries distributed  i.i.d.\ $Ber(0.1),Ber(0.4),Ber(0.7),Ber(0.8)$, respectively. The red curve indicates the density of $\tfrac{2}{3} (\eta-1)$, where $\eta$ follows the \MP law with parameter $\gamma=800/1200$, which is the limiting spectral distribution of $\bfT$ for $p/n\to \gamma$ (part $(2)$ of Theorem~\ref{thm:kendall}).}
  \label{fig:kendall_mp}
\end{figure}

\begin{figure}[htb!]
  \centering
  \includegraphics[trim=2cm 3cm 2cm 3cm,
  clip,width=\linewidth]{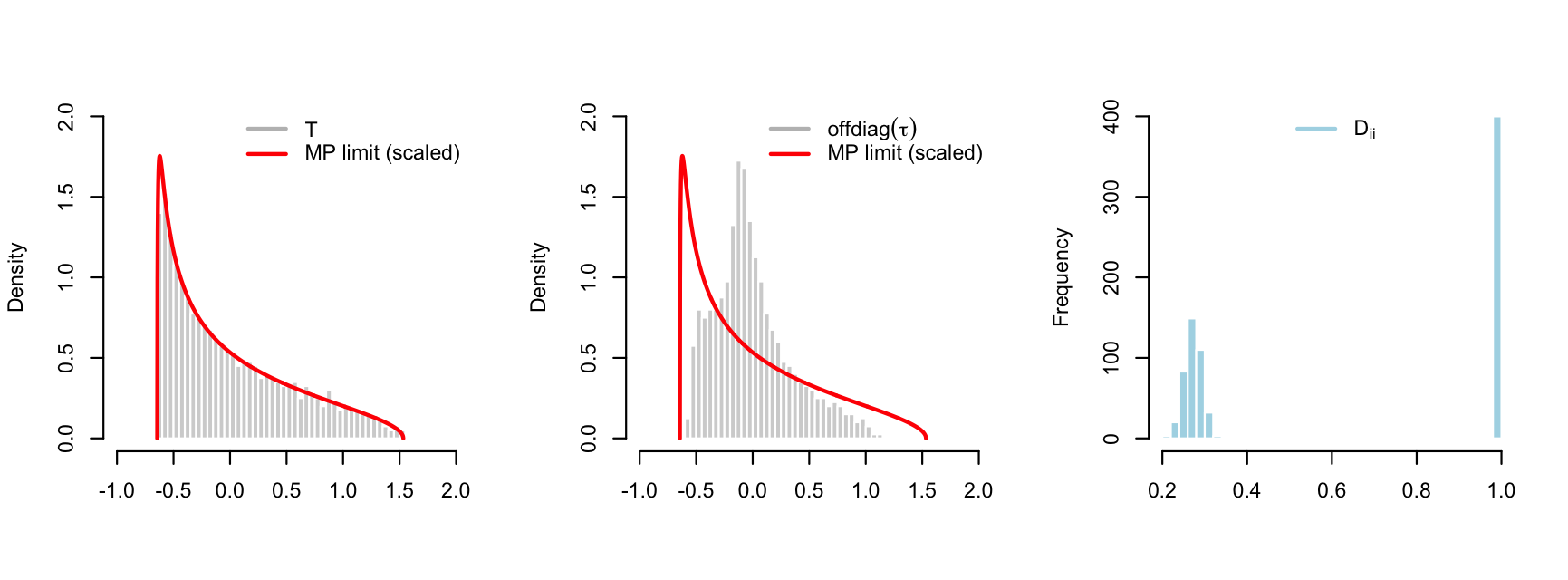}
   \includegraphics[trim=2cm 3cm 2cm 3cm,
  clip,width=\linewidth]{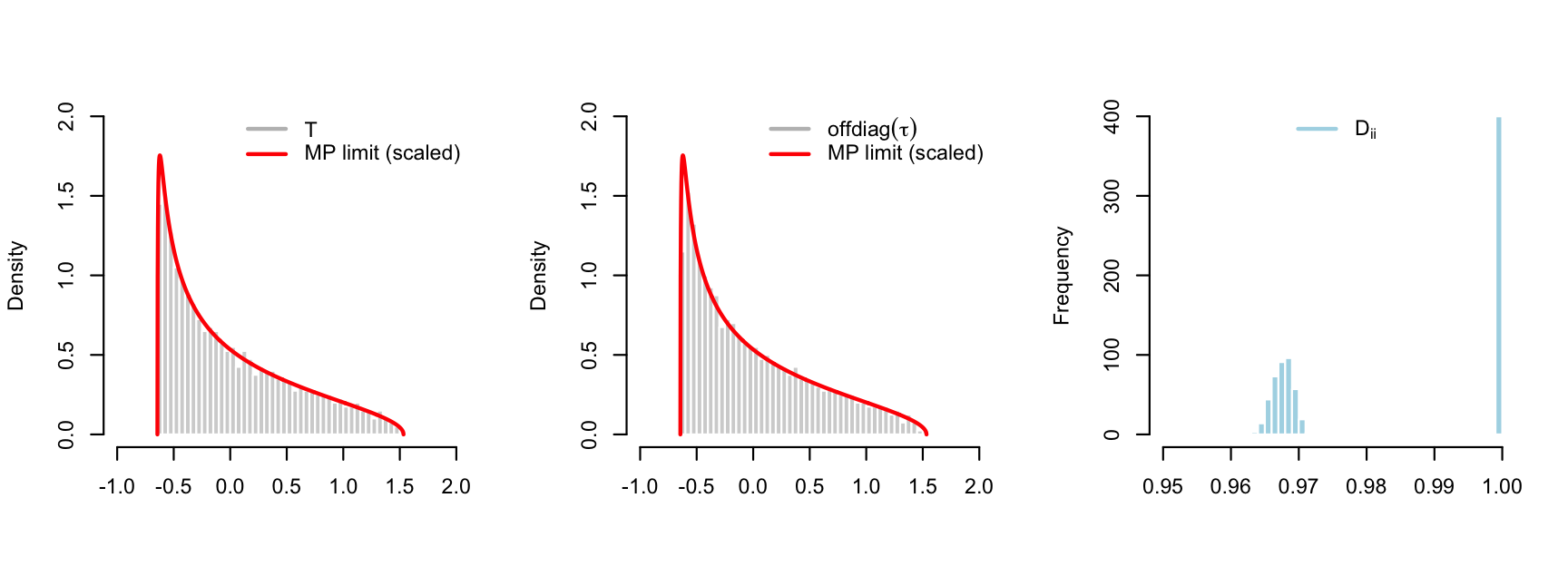}
  \caption{Normalized histograms of the simulated eigenvalues of $\bfT$ (left panels), $\offdiag(\boldsymbol{\tau})$ (middle panels) and the diagonal entries of $\bfD$ (right panels), where $(p,n)=(800,1200)$. In the top row, every second row of $\bfX$ has entries distributed $\operatorname{Poi}(3)$ and $\operatorname{Exp}(1)$, respectively. In the bottom row, every second row of $\bfX$ has entries as distributed $\operatorname{Ber}(0.1)$ and $\operatorname{Exp}(1)$, respectively.  The red curve indicates the density of $\tfrac{2}{3} (\eta-1)$, where $\eta$ follows the \MP law with parameter $\gamma=800/1200$, which is the limiting spectral distribution of $\bfT$ for $p/n\to \gamma$ (part $(2)$ of Theorem~\ref{thm:kendall}). }
  \label{fig:kendallDiscCont}
\end{figure}

In the following remark, we comment on the role of the scaling matrix $\bfD$.
Figure~\ref{fig:kendall_scl}, Figure~\ref{fig:kendall_mp} and Figure~\ref{fig:kendallDiscCont} show the effect of the correct normalization of $\offdiag(\boldsymbol\tau)$, using the diagonal matrix $\bfD$.  

\begin{remark}\label{remark_scaling_D}{\em
    The scaling based on $\bfD$ is crucial for obtaining a pivotal limiting spectral distribution in the presence of ties. Here, we use the term \textit{pivotal} or \textit{universal} distribution to refer to a distribution which does not depend on the initial distribution of the entries of $\bfX.$ Hence, we call the asymptotic results obtained in this paper \textit{distribution-free.} 
    
    To see this, we concentrate on the case of a data matrix $\bfX$ consisting of i.i.d.\ entries  with generic element $X$ that follows a discrete distribution. Denote its c.d.f. and p.m.f. as $H$ and $p_X$, respectively, that is, for $x\in\R$, we define
    \begin{align*}
        H(x) := \P (X \leq x) \quad \text{ and } \quad p_X(x) := \P(X=x). 
    \end{align*}
   For convenience, we concentrate on the case $\gamma=0$. From the proof of Theorem \ref{thm:kendall}, it can be deduced that the LSD of $\sqrt{n/p} \,  \offdiag (\boldsymbol{\tau} )$ is 
   \begin{align} \label{lsd}
       \left[ 3\, \Var ( 2 H(X) - p_X(X)) \right] \frac{2}{3} \zeta .
   \end{align}
Note that the LSD in \eqref{lsd} is pivotal if we restrict ourselves to continuous distributions for $X$ (see Remark \ref{rem_con_kendall}). In this case we have $2 H(X) - p_X(X)\sim U(0,2)$, so that $\Var ( 2 H(X) - p_X(X))=1/3$ does not depend on the specific distribution of $X$ provided that it is continuous.
    To explain the effect of $\Var ( 2 H(X) - p_X(X))$ in \eqref{lsd} for discrete distributions, consider for instance an i.i.d.\ data matrix $\bfX$ with generic element $X\sim\operatorname{Ber}(m), m\in(0,1)$. 
    Then, it is straightforward to verify that $\Var ( 2 H(X) - p_X(X))=\Var(1-m+X) = m(1-m)$. Thus, \eqref{lsd} is not universal, whereas the limiting spectral distribution of $\sqrt{n/p}\, \bfT$ is universal.  In fact, \eqref{lsd} indicates that if all rows possess the same distribution, then the $D_{ii}$'s converge to $3\Var ( 2 H(X) - p_X(X))$ which is $3/4$ in the $\operatorname{Ber}(1/2)$ case; see the right panel of Figure~\ref{fig:kendall_scl} for an illustration. Similar observations hold in the case $\gamma >0$, but for the sake of brevity, we omit a full discussion here.

However, it is noteworthy that the effect of $\bfD$ is in general not close to a simple scaling. This can be observed in Figure~\ref{fig:kendall_mp}. Here, the spectra are not related by scaling.  This is also clearly visible in the top row of Figure~\ref{fig:kendallDiscCont}, in which we simulate ensembles containing both discrete ($\operatorname{Ber}(0.1)$) and continuous data ($\operatorname{Exp(1)}$). The bottom row in 
Figure~\ref{fig:kendallDiscCont} illustrates the effect of replacing the $\operatorname{Ber}(0.1)$ rows by $\operatorname{Poi}(3)$ rows, while keeping the $\operatorname{Exp(1)}$ rows unchanged. Since the Poisson distribution generates substantially fewer ties, the empirical spectral distribution of $\bfT$ becomes much closer to the classical limit associated with continuous entries. However, the spectral distribution of $\offdiag(\boldsymbol{\tau})$ remains different from the fully continuous case, as illustrated by the histogram of $\mathbf{D}$ in the bottom-right panel of the figure.
} \end{remark}

\subsection{Rows on the Euclidean unit sphere} \label{sec_unit_sphere}
In this subsection, we study the spectral distribution of general random matrices whose rows live on the Euclidean unit sphere and satisfy mild moment conditions. This will be discussed in detail in Remark \ref{rem_mom_cond} which shows that these moment conditions are sharp. The result presented in this subsection facilitates the proof of Theorem \ref{thm:spearman}. 
Moreover, our study continues the line of research on the eigenvalue distribution of the sample correlation, and this example will be discussed further after the formulation of the main theorem in this subsection. 

\begin{theorem}[Unit sphere] \label{thm_unit_sphere}
Let $\bfY$ be a $p\times n$ random matrix with independent rows $\mathbf{y}_k=(Y_{k1},\ldots, Y_{kn})$, $1\leq k \leq p$. For all $1 \leq k \leq p$, suppose that $\mathbf{y}_k$ is on the Euclidean unit sphere, and that $\E[Y_{ki}^2]=\E[Y_{kj}^2]$, $\E[Y_{ki}^4]=\E[Y_{kj}^4]$, $\E[Y_{ki}Y_{kj}]= \E[Y_{k1}Y_{k2}]$ and $\E[Y_{ki}^2 Y_{kj}^2]= \E[Y_{k1}^2 Y_{k2}^2]$  for $1\le i<j\le n$. Further, we impose
\begin{equation} \label{eq_mom_cond_unit_sphere}
\frac{n^2}{p^2} \sum_{i=1}^p \E \big[ Y_{i1}^4 \big] = o(1) \quad \text{and} \quad \frac{n^2}{p^{3/2}} \sum_{i=1}^p |\E\big[ Y_{i1}Y_{i2}\big]|= o(1)\,,\qquad \pto\,.
\end{equation}
\begin{enumerate}
    \item[(1)]
    Assume $p/n\to 0$ as $p\to\infty$. Then almost surely, the limiting spectral distribution of $\sqrt{n/p} \, (\bfY \bfY'-\bfI)$ is the semicircle law $G$.
\item[(2)]
    Assume $p/n\to \gamma >0 $ as $p\to\infty$. 
Then almost surely, the limiting spectral distribution of $\bfY\bfY'$ is the \MP distribution $F_{\gamma}$.
\end{enumerate}
\end{theorem}
One prominent application of Theorem \ref{thm_unit_sphere} regards the spectral analysis of the (uncentered) sample correlation matrix $\bfR :=\{\diag(\X\X')\}^{-1/2}\, \X\X'\{\diag(\X\X')\}^{-1/2}$, where the data matrix $\X$ is as defined in \eqref{eq:data} with non-degenerate row distributions.
The $(i,j)$-entry of $\bfR$ is given by 
\begin{equation}\label{eq:corrR}
R_{ij}= 
 \sum_{t=1}^n Y_{it}Y_{jt}\,, \qquad i,j=1,\ldots,p\,,
\end{equation} 
where
\begin{align} \label{eq_def_y_corr}
Y_{it}:=X_{it}/\sqrt{X_{i1}^2+\cdots + X_{in}^2 }\,, \qquad 
i=1,\ldots,p;\; t=1,\ldots,n\,.
\end{align}
Thus, the matrix $\Y=(Y_{ij})_{p\times n}$ has independent rows on the Euclidean unit sphere.

If the true means $\E[X_{i1}]$ of the data generating process are zero (or known), one can work with the above definition of $\bfR$.  However, in general, one needs to estimate the population mean vector. To this end, one typically considers the centered sample correlation matrix
\begin{equation*}
  \bfR_c:= \{\diag((\X - \bar{\x})(\X - \bar{\x})')\}^{-1/2} \,(\X - \bar{\x})(\X - \bar{\x})' \, \{\diag((\X - \bar{\x})(\X - \bar{\x})')\}^{-1/2}\,,
\end{equation*}
where $\bar{\x}=n^{-1}\X\mathbf{1}$ is the sample mean
and $\mathbf{1}=(1,\ldots,1)'$ denotes the $n$-dimensional vector of ones.
By definition, we have $\bfR_c=\widetilde{\bfY}\widetilde{\bfY}'$, where the matrix $\widetilde{\bfY}=(\widetilde{Y}_{ij})_{p\times n}$ has entries 
\begin{equation}\label{eq:defytilde}
\widetilde{Y}_{ij}:=\widetilde{X}_{ij}/\sqrt{\widetilde{X}_{i1}^2+\cdots \widetilde{X}_{in}^2 } \quad \text{ with } \widetilde{X}_{ij}:=X_{ij}-\frac{1}{n} \sum_{t=1}^n X_{it}\,,
\end{equation}
and independent rows on the Euclidean unit sphere.

 In the case $\E[X_{i1}^4]<\infty$, many results about sample correlation matrices can reduced to the covariance case via a comparison of their spectra (see \cite{heiny:2022}). However, this approximation technique breaks down if the fourth moment is no longer finite. 
 The literature of exploring the limiting eigenvalue distribution beyond the common fourth moment condition is rather scarce. Recent works in this direction are \cite{doernemann:heiny:2025, heiny:parolya:2024, heiny:yao:2020}, where the case $\gamma >0$ is treated under mild moment assumptions. Theorem \ref{thm_unit_sphere} applied to $\bfR$ and $\bfR_c$ continues this line of literature and provides the limiting spectral distribution under a moment assumption, which is less restrictive than the common fourth moment condition $\E [ X_{i1}^4 ] <\infty$, as discussed below. To the best of our knowledge, the centered sample correlation matrix has not been studied under infinite fourth moments.  

\begin{theorem}[Sample correlation] \label{thm:samplecorr}
Let the data matrix $\X$ be defined as in \eqref{eq:data} with non-degenerate row distributions.
\begin{enumerate}
    \item[(1)]
    Assume $p/n\to 0$ as $p\to\infty$.
If \eqref{eq_mom_cond_unit_sphere} holds,
then, almost surely, the limiting spectral distribution of $\sqrt{n/p} \, (\bfR-\bfI)$ is the semicircle law $G$.\\
Moreover, if 
\begin{equation} \label{eq:moma}
\frac{n^2}{p^2} \sum_{i=1}^p \E \big[ \widetilde{Y}_{i1}^4 \big] = o(1)\,,\qquad \pto\,,
\end{equation}
then, almost surely, the limiting spectral distribution of $\sqrt{n/p} \, (\bfR_c-\bfI)$ is the semicircle law~$G$.
\item[(2)]
    Assume $p/n\to \gamma >0 $ as $p\to\infty$. 
If \eqref{eq_mom_cond_unit_sphere} holds,
then, almost surely, the limiting spectral distribution of $\bfR$ is the \MP distribution $F_{\gamma}$.\\
Moreover, if \eqref{eq:moma} holds,
then, almost surely, the limiting spectral distribution of $\bfR_c$ is the \MP distribution $F_{\gamma}$.
\end{enumerate}
\end{theorem}
\begin{proof}
The statements for $\bfR$ are a direct application of Theorem \ref{thm_unit_sphere}. Regarding $\bfR_c$, we note that $\sum_{t=1}^n \widetilde{Y}_{it}=0$ by construction. 
In conjunction with $\sum_{j=1}^n \widetilde{Y}_{ij}^2 = 1$ for $1 \leq i \leq p$, it is easy to check that $\E[\widetilde{Y}_{i1} \widetilde{Y}_{i2}]=-1/(n(n-1))$ and therefore we conclude
\begin{align*}
 \frac{n^2}{p^{3/2}} \sum_{i=1}^p |\E [ \widetilde{Y}_{i1} \widetilde{Y}_{i2} ]|=\frac{n^2}{p^{3/2}}  \cdot \frac{p}{n(n-1)} =o(1)\,, \qquad \pto\,,
\end{align*}
which shows that \eqref{eq:moma} implies \eqref{eq_mom_cond_unit_sphere}. An application of Theorem \ref{thm_unit_sphere} concludes the proof of the statements for $\bfR_c$.   
\end{proof}

The following table summarizes the conditions for the limiting spectral distributions of Spearman's $\rho$, Kendall's $\tau$ and sample correlation matrices, respectively.

\begin{tabular}{|l||c | c |c|}
\hline
Matrix: & $\SR$ (Spearman's $\rho$) & $\bfT$ (Kendall's $\tau$) & $\bfR$ (Sample correlation) \\
\hline
Statement: & Theorem \ref{thm:spearman} & Theorem \ref{thm:kendall} & Theorem \ref{thm:samplecorr} \\
\hline
Assumption: & Assumption \ref{ass_asympt_non_degen} & Assumption \ref{ass_asympt_non_degen} & $\frac{n^2}{p} \E[Y_{11}^4] =o(1)$ \\
\hline
\end{tabular}

\begin{remark} \label{rem_mom_cond}
{\em 
To provide a better understanding of Theorem \ref{thm:samplecorr} on the limiting spectral distribution of the sample correlation matrices $\bfR$ and $\bfR_c$, we will shed light on conditions \eqref{eq_mom_cond_unit_sphere} and \eqref{eq:moma} when $Y_{it}$ and $\widetilde Y_{it}$  follow the definitions \eqref{eq_def_y_corr} and \eqref{eq:defytilde}, respectively. In the following, we assume that the matrix $(X_{it})_{1 \leq i \leq p, 1\leq t \leq n}$ has i.i.d. entries with generic element $X$ that satisfies $\E[X^2]=1$ if $\E[X^2]<\infty$. Thus, condition \eqref{eq_mom_cond_unit_sphere} reduces to
\begin{align} \label{eq_mom_cond_y}
    \frac{n^2}{p} \E[Y_{11}^4] =o(1) \quad \text { and } \quad \frac{n^2}{p^{1/2}} \E[Y_{11}Y_{12}]=o(1)\,, \qquad \pto\,.
\end{align}

\begin{enumerate}
    \item
To begin with, let $\gamma>0$. In this case, the first moment condition in \eqref{eq_mom_cond_y} is equivalent to
\begin{equation}\label{eq:ny4}
\lim_{\pto} n\,\E [ Y_{11}^4 ] =0\,.
\end{equation}
Condition \eqref{eq:ny4} can be identified as a necessary condition for the convergence of $F_{\bfR}$ to the \MP law (see Theorem 3.1 in \cite{heiny:mikosch:2017:corr}), and therefore, in this setting, the proposed moment condition in Theorem \ref{thm_unit_sphere} is minimal. 

Furthermore, we know from \cite[Theorem~2.4]{doernemann:heiny:2025} that \eqref{eq:ny4} is equivalent to $X$ being in the domain of attraction of the normal distribution, which in turn implies by \cite{gine:goetze:mason:1997} that
\begin{align}\label{eq:hfdhdrs} 
    n^{2} \E [ Y_{11} Y_{12} ] = o(1). 
\end{align}

\item Next, we assume that $\gamma=0$. In this case, we notice that the validity of \eqref{eq_mom_cond_y} depends on the rates with which $p/n$ and $n\E [ Y_{11}^4 ]$ tend to $0$. Since \eqref{eq:ny4} is necessary for \eqref{eq_mom_cond_y}, we deduce from \eqref{eq:hfdhdrs} that the second moment condition in \eqref{eq_mom_cond_y} can be dropped without loss of generality. 

In turn, this implies that condition \eqref{eq_mom_cond_y} is satisfied for all distributions in the domain of attraction of the normal law if $p$ grows sufficiently fast; see Example \ref{ex:grp} for more details.
\item Finally, we turn to condition \eqref{eq:moma} concerning the centered sample correlation matrix $\bfR_c$, which can be written as
\begin{align}\label{eq:tildejkjk}
    \frac{n^2}{p} \E[\widetilde{Y}_{11}^4] 
		= \frac{n}{p} \E\left[ \frac{\sum_{j=1}^n \big(X_{1j}-\bar{x}_1\big)^4}{\Big(\sum_{j=1}^n \big(X_{1j}-\bar{x}_1\big)^2\Big)^2} \right]
		=o(1)\,, \qquad \pto\,,
\end{align}
where $\bar{x}_1=\frac{1}{n} \sum_{t=1}^n X_{1t}$. We note that the ratio in the expectation is bounded by one. Without loss of generality we assume $\E[X]=0$ if $\E[|X|]$ is finite. Then one can show that 
\begin{align*}
\sum_{j=1}^n \big(X_{1j}-\bar{x}_1\big)^4&= \sum_{j=1}^n X_{1j}^4 \, (1+o_{\P}(1)),\\
\sum_{j=1}^n \big(X_{1j}-\bar{x}_1\big)^2&= \sum_{j=1}^n X_{1j}^2 \, (1+o_{\P}(1));
\end{align*}
see \cite[p.~5]{cohen:davis:samorodnitsky:2020} for a proof in the heavy-tailed case. In conjunction with the above mentioned boundedness of the ratio, it follows $\E[\widetilde{Y}_{11}^4]\sim \E[Y_{11}^4]$, which establishes that condition \eqref{eq:tildejkjk} is equivalent to the first moment condition in \eqref{eq_mom_cond_y}.
\end{enumerate}
}\end{remark}

\begin{example}\label{ex:grp}{\em
We provide further examples satisfying the condition $\tfrac{n^2}{p} \, \E[Y_{11}^4]=o(1)$ in the setting of Remark~\ref{rem_mom_cond}. To this end, we use tools from \cite{fuchs:joffe:teugels:2001} to analyze the asymptotic behavior of $\E[Y_{11}^4]$. We know from Theorem 4.1 therein that 
\begin{equation}\label{eq:limsup}
\limsup_{\nto} \frac{n^2\,\E [ Y_{11}^4 ]}{\int_0^n x \P(X^2>x) \,\dint x} \le 2\,.
\end{equation}
The denominator in the above display may be interpreted as a truncated fourth moment of $X$. Indeed, we obtain via partial integration that
$\E[X^4]=2 \int_0^{\infty} x \P(X^2>x) \,\dint x.$
We consider different cases.
\begin{itemize}
\item If $\E[X^4]<\infty$, Theorem 3.1 in \cite{fuchs:joffe:teugels:2001} asserts that $n^2\,\E [ Y_{11}^4 ]\to \E[X^4]$. Therefore, condition \eqref{eq_mom_cond_y} holds.
\item Now let us assume that $X$ is regularly varying with index $\alpha\in [2,4)$, that is 
\begin{align*}\label{eq:regvar}
\P(|X|>x)=\dfrac{L(x)}{x^{\alpha}}\,,\qquad x>0,
\end{align*}
for a function $L$ that is slowly varying at infinity. Thus, regularly varying distributions possess power-law tails and moments of $|X|$ of higher order than $\alpha$ are infinite. Typical examples include the Pareto distribution with parameter $\alpha$ and the $t$-distribution with $\alpha$ degrees of freedom.

For regularly varying $X$ with index $\alpha\in [2,4)$, Theorem 5.1 in \cite{fuchs:joffe:teugels:2001} states that
\begin{equation*}
\lim_{\nto} \frac{\E [ Y_{11}^4 ]}{\P(X^2>n)} = \Gamma(2-\alpha/2) \Gamma(1+\alpha/2)\,.
\end{equation*}
Using a Potter bound argument, we conclude that $n^2\,\E \big[ Y_{11}^4 \big]\le n^{2-\alpha/2+\vep}$
for any $\vep>0$ provided that $n$ is sufficiently large.
Therefore, if $n=p^\delta \ell(n)\to \infty$ for some slowly varying function $\ell$ and  $\delta \in [1,\infty)$, then condition \eqref{eq_mom_cond_y} holds for all regularly varying distributions $X$ with index $\alpha>4-2/\delta$.
\end{itemize}
}\end{example}

\section{Proofs of Theorem \ref{thm:spearman} and Theorem \ref{thm_unit_sphere}}
\label{sec_proof_spearman_unit_sphere}
In this section, we provide the proofs of Theorem \ref{thm:spearman} and Theorem \ref{thm_unit_sphere}. 
Section \ref{sec:proof_lsd} contains the proof of Theorem \ref{thm_unit_sphere} and the analogue assertion for the Spearman's rank correlation matrix is proven in Section~\ref{sec_proof_rho}. Both proofs rely on the Stieltjes transform, while a different strategy is employed for the proof of Theorem \ref{thm:kendall} in Section~\ref{sec_proof_kendall}. Before turning to the proofs, we start with some preparations.

\subsection*{Preliminaries}
In the following, the spectral (or operator) norm $\twonorm{\bfC}$ is the square root of the largest eigenvalue of $\bfC\bfC'$ for any matrix $\mathbf{C}$. Moreover, we will use the notation  $a_p \lesssim b_p$ for sequences $(a_p)_p$ and $(b_p)_p$ of real numbers if there exists a positive constant $c$ not depending on $p$ such that $a_p \le c\, b_p$ for sufficiently large $p$.
Note that while the constant $c$ is not allowed to vary with $p\in\N$, it may depend on $z\in\mathbb{C}^+:=\{q\in \C: \Im(q)>0\}$.

For identifying the limiting spectral distribution, an extremely useful tool is the {\em Stieltjes transform}
\begin{equation*}
s_{\bfA}(z)= \int_{\R} \frac{1}{x-z} \dint F_{\bfA}(x) = \frac{1}{p} \tr(\bfA -z \bfI)^{-1}\,, \quad z\in \C^+\,
\end{equation*}
of the empirical spectral distribution $F_{\bfA}$,
where $\C^+$ denotes the complex numbers with positive imaginary part and $\bfA$ is a $p\times p$ matrix with real eigenvalues. 
Weak convergence of $(F_{\bfA_n})$ to $F$ is equivalent to $s_{F_{\bfA_n}}(z) \to s_F(z)$ a.s.~for all $z\in \C^+$
(see \cite[Chapter~3]{bai:silverstein:2010} or \cite{bai:fang:liang:2014,yao:zheng:bai:2015}).
The semicircle law $G$, defined in \eqref{eq:semicircle}, has Stieltjes transform
\begin{equation} \label{eq_stieltjes_semicircle}
s_G(z)=\frac{\sqrt{z^2-4} -z}{2}\,.
\end{equation}
The Stieltjes transform of the \MP law $F_{\gamma}$, defined in \eqref{eq:MPch1}, is given by
  \begin{align} \label{eq_stieltjes_mp}
        s_{F_{\gamma}}(z)
        & =  \frac{1 - \gamma - z + \sqrt{( 1- \gamma - z)^ 2 - 4 \gamma z }}{2\gamma z}.
    \end{align}

\subsection{Proof of Theorem \ref{thm_unit_sphere}} \label{sec:proof_lsd}

In view of the discussion below Theorem \ref{thm_unit_sphere}, we use the notation $\bfR=\Y\Y'$ with $\Y$ as described in Theorem \ref{thm_unit_sphere}. Let $\gn:= p/n$. In order to show Theorem \ref{thm_unit_sphere} part (1), we will prove that the Stieltjes transform of $\gn^{-1/2} (\bfR-\bfI)$ converges to the Stieltjes transform $s_G$ of the semicircle distribution $G$ almost surely. This implies weak convergence of $F_{\gn^{-1/2} (\bfR-\bfI)}$ to $G$ almost surely (see, e.g., \cite{fleermann:kirsch:2023}).
For $z\in \C^+$, we proceed in 2 steps. 
\begin{itemize}
\item[(1)] In Lemma \ref{lem:expectedtransform}, we show that $s_{\gn^{-1/2} (\bfR-\bfI)}(z) -  \E[s_{\gn^{-1/2} (\bfR-\bfI)}(z)]\to 0$ \as\ for $\pto$.
\item[(2)] In Lemma \ref{lem:convexp}, we show $\E[s_{\gn^{-1/2} (\bfR-\bfI)}(z)] \to s_G(z)$ for $\pto$.
\end{itemize}
Finally, Lemma \ref{lem_mp} considers the \MP regime $\gamma >0$ in Theorem \ref{thm_unit_sphere} part (2) and its proof will be a subtle modification of the ideas for the case $\gamma = 0$.

\begin{lemma}\label{lem:expectedtransform}
Assume $\gn \to 0$.
Then
\begin{equation*}
s_{\gn^{-1/2} (\bfR-\bfI)}(z) - \E[s_{\gn^{-1/2} (\bfR-\bfI)}(z)] \cas 0\,, \qquad \pto\,, z\in \C^+\,.
\end{equation*}
\end{lemma}
\begin{proof}
We use the fact that
$\bfR=\Y\Y'$ and $\Y'\Y$ have the same non-zero eigenvalues with the same multiplicities. 
Since $\lambda_{i}(\Y'\Y)=0$ whenever $i> p$ 
we obtain a connection between the two spectral distributions:
\begin{equation*}
F_{\Y'\Y}(x)=\Big(1-\frac{p}{n}\Big) \1_{[0,\infty)}(x) + \frac{p}{n} F_{\Y\Y'}(x)\,,\qquad x\in \R\,.
\end{equation*}
Hence, 
\begin{align}
s_{\bfR}(z)&= \int \frac{1}{x-z} \dint F_{\bfR}(x)\nonumber\\
&= \int \frac{1}{x-z} \dint \Big(\frac{n}{p} F_{\Y'\Y} - \Big( \frac{n}{p}-1 \Big) \1_{[0,\infty)} \Big)(x)\nonumber\\
&= \frac{n}{p} s_{\Y'\Y}(z)- \Big( \frac{n}{p}-1 \Big) \frac{1}{-z}\,,\qquad z\in\C^+\,. 
\label{eq:sgdsglop}
\end{align}
Note that $s_{\bfR - \bfI} (z) = s_{\bfR} (z +1)$ and for a constant $c\neq 0$, we have $s_{c\bfA}(z)=c^{-1} s_{\bfA}(c^{-1}z)$.
Therefore, we obtain
\begin{equation}
\label{eq:n1}
\begin{split}
s_{\gn^{-1/2} (\bfR-\bfI)}(z)&= \gn^{1/2} s_{\bfR}(1+\gn^{1/2} z)
= \gn^{1/2} \Big( \gn^{-1} s_{\Y'\Y}(1+\gn^{1/2} z) + \frac{\gn^{-1}-1}{1+\gn^{1/2} z} \Big)\,.
\end{split}
\end{equation}
So it remains to show
\begin{equation}\label{eq:gesgdg}
\gn^{-1/2} ( s_{\Y'\Y}(1+\gn^{1/2} z)-\E[s_{\Y'\Y}(1+\gn^{1/2} z)]) \cas 0\,, \quad \pto\,, z\in \C^+\,.
\end{equation}
We proceed analogously to the proof of Lemma 6 in \cite{elkaroui:2009}. Let $\y_i\in \R^{1\times n}, 1\le i\le p$, denote the rows of $\Y$ and note that $\Y'\Y= \sum_{i=1}^p \y_i' \y_i$. Write $M_k=\Y'\Y -\y_k' \y_k$ and let $\mathcal{F}_j$ denote the filtration generated by $\y_i,1\le i\le j$, and we define $\mathcal{F}_ 0 = \{ \emptyset \}$. Since $M_k$ is independent of $\y_k$, we have
\begin{equation}\label{eq:dsgfds}
\E\big[\tr((M_k-z\bfI)^{-1})| \mathcal{F}_k\big]=\E\big[\tr((M_k-z\bfI)^{-1})| \mathcal{F}_{k-1}\big]\,, \quad z\in \C^+\,.
\end{equation}
We will write the \lhs~of \eqref{eq:gesgdg} as a sum of martingale differences. For simplicity of notation we write $q=1+\gn^{1/2} z$, $s=s_{\Y'\Y}$ and denote the imaginary part of $z$ by $v$.
For $z\in \C^+$ we have
\begin{equation}\label{eq:martdiff}
s(q)-\E[s(q)]= \sum_{k=1}^p \left\{ \E[s(q)| \mathcal{F}_k]-\E[s(q)| \mathcal{F}_{k-1}] \right\} ,
\end{equation}
and therefore,
\begin{equation*}
\begin{split}
|s(q)-\E[s(q)]|
&\le \sum_{k=1}^p \left[ \big| \E\big[ s(q) -\tfrac{1}{n} \tr((M_k-q\bfI)^{-1}) | \mathcal{F}_k  \big]\big|+  \big| \E\big[ s(q) -\tfrac{1}{n} \tr((M_k-q\bfI)^{-1}) | \mathcal{F}_{k-1}  \big]\big|\right]\\
&\le \sum_{k=1}^p \frac{2}{n v \gn^{1/2}}.
\end{split}
\end{equation*}
Here, the first inequality is due to the triangle inequality and \eqref{eq:dsgfds}, and the last inequality follows from Lemma 2.6 in \cite{silverstein:bai:1995}, which allows concluding $|s(q)-n^{-1}\tr(M_k-q\bfI)^{-1}|\leq n^{-1}v^{-1}\gamma_p^{-1/2}$. Thus $s(q)-\E[s(q)]$ and also its real and imaginary parts are a sum of bounded martingale differences. By Azuma's inequality (see Lemma 4.1 in \cite{ledoux:2001}), we get
\begin{equation*}
\begin{split}
\P(\gn^{-1/2}|s(q)-\E[s(q)]|>\vep)& \le 4 \exp\Big( -\frac{\vep^2 \gn^2 n^2 v^2}{32 p}\Big) = 4 \exp\Big( -\frac{\vep^2 p v^2}{32}\Big)\,, \quad \vep >0\,.
\end{split}
\end{equation*}
This implies \eqref{eq:gesgdg} via the first Borel--Cantelli lemma. 
\end{proof}

\begin{lemma}\label{lem:convexp}
Assume $\gn \to 0$.
Then
\begin{equation*}
\E[s_{\gn^{-1/2} (\bfR-\bfI)}(z)] \to s_G(z)\,, \qquad \pto\,, z\in \C^+\,,
\end{equation*}
where the Stieltjes transform $s_G$ of the semicircle law is defined in \eqref{eq_stieltjes_semicircle}.
\end{lemma}
\begin{proof}
By \eqref{eq:n1}, it suffices to prove that
\begin{equation}\label{eq:dsgs}
\lim_{\pto} \gn^{1/2} \E\big[s_{\bfR}(1+\gn^{1/2} z)\big] = \frac{\sqrt{z^2-4} -z}{2} = s_G(z)\,, \quad z\in \C^+\,.
\end{equation}
Let $v= \Im(z)>0$ and  $q= 1+\gn^{1/2} z \in \mathbb{C}$. In what follows, we will denote by $\overline{q}$ the complex conjugate of $q$. 
By Theorem A.4 in \cite{bai:silverstein:2010}, we have
\begin{equation*}
s_{\bfR}(q)= \frac{1}{p} \sum_{k=1}^p \frac{1}{\y_k \y_k'-q-\y_k \Y_{-k}' (\Y_{-k}\Y_{-k}'-q\bfI_{p-1})^{-1} \Y_{-k} \y_k'}\,,
\end{equation*}
where $\y_k\in \R^{1\times n}$ are the rows of $\Y$ and $\Y_{-k}\in \R^{(p-1)\times n}$ is the matrix $\Y$ with the $k$th row removed. In \cite[pages 55 and 56]{bai:silverstein:2010}, it is shown that 
\begin{equation}\label{eq:sols1}
\E[s_{\bfR}(q)]= \frac{1}{2\gn q} \left(1-q-\gn+\gn q \delta_p+\sqrt{(1-q-\gn-\gn q \delta_p)^2-4 \gn q}\right)\,,
\end{equation}
where 
\begin{equation*}
\begin{split}
\delta_p &= -\frac{1}{p} \sum_{k=1}^p \E\bigg[ \frac{\vep_k}{(1-q-\gn -\gn q \E[s_{\bfR}(q)]) (1-q-\gn -\gn q \E[s_{\bfR}(q)]+\vep_k)}  \bigg]\,, \\
\vep_k &= \y_k \y_k'-1-\y_k \Y_{-k}' (\Y_{-k}\Y_{-k}'-q\bfI_{p-1})^{-1} \Y_{-k} \y_k'+\gn+\gn q\E[s_{\bfR}(q)]\,.
\end{split}
\end{equation*}
Note that $\y_k \y_k'-1=0$ in the definition of $\vep_k$.
Using \eqref{eq:sols1}, $\gn\to 0$ and the definition of $q$, a straightforward calculation yields
\begin{equation*}
\gn^{1/2}\E[s_{\bfR}(q)]\sim \frac{-z+ z \gn\delta_p + \gn^{1/2} \delta_p +\sqrt{z^2-4+2\gn \delta_p +\gn \delta_p^2 + 2 \delta_p \gamma_p^{1/2} z }}{2}\,, \quad \pto\,.
\end{equation*}
Therefore, we obtain \eqref{eq:dsgs} provided 
\begin{equation}\label{eq:grsgrsd}
\lim_{\pto} \gn^{1/2} \delta_p =0\,.
\end{equation}
Thus, it remains to prove \eqref{eq:grsgrsd} in order to obtain \eqref{eq:dsgs}. To this end, we write
\begin{align}
\delta_p &= -\frac{1}{p} \sum_{k=1}^p  \frac{\E[\vep_k]}{(1-q-\gn -\gn q \E[s_{\bfR}(q)])^2 }  \nonumber \\
& \quad +\frac{1}{p} \sum_{k=1}^p \E\bigg[ \frac{\vep_k^2}{(1-q-\gn -\gn q \E[s_{\bfR}(q)])^2 (1-q-\gn -\gn q \E[s_{\bfR}(q)]+\vep_k)}  \bigg] \nonumber \\
 &=: J_1+J_2\,. \label{eq_def_J}
\end{align}

Turning to analysis of $\vep_k$, since $\y_k$ and $\Y_{-k}$ are independent, we have
\begin{equation}\label{eq:1234}
\begin{split}
\E \big[&\y_k \Y_{-k}' (\Y_{-k}\Y_{-k}'-q\bfI_{p-1})^{-1} \Y_{-k} \y_k'\big]= 
\tr \Big( \E\big[ (\Y_{-k}\Y_{-k}'-q\bfI_{p-1})^{-1} \Y_{-k}  \E[\y_k'\y_k] \Y_{-k}' \big]\Big)\\
&= \big( \tfrac{1}{n} -\E[Y_{k1}Y_{k2}]\big) \Big \{ p-1 +q \E\Big[ \tr \big( (\Y_{-k}\Y_{-k}'-q\bfI_{p-1})^{-1} \big) \Big] \Big\}\\
& \quad + \E[Y_{k1}Y_{k2}] \E\Big[ \tr \big( (\Y_{-k}\Y_{-k}'-q\bfI_{p-1})^{-1} \Y_{-k} 1_n1_n' \Y_{-k}' \big) \Big]\,,
\end{split}
\end{equation}
 where 
\begin{align} \label{mean_yty}
\E[\y_k'\y_k]= (n^{-1} -\E[Y_{k1}Y_{k2}]) \bfI_n +\E[Y_{k1}Y_{k2}] 1_n 1_n',
\end{align} 
and $1_n=(1,\ldots,1)'\in \R^n$ were used for the last equality. 
From the definition of $\vep_k$ and \eqref{eq:1234} we obtain
\begin{equation}\label{eq:1236}
\begin{split}
\E[\vep_k]&= \Big(\frac{p}{n}-\frac{p-1}{n}\Big) + \frac{q}{n}\Big(p\E[s_{\bfR}(q)] -\E\Big[ \tr \big( (\Y_{-k}\Y_{-k}'-q\bfI_{p-1})^{-1} \big) \Big]\Big)+\E[Y_{k1} Y_{k2}] (p-1)\\
 &\quad +\E[Y_{k1}Y_{k2}] q \E\Big[ \tr \big( (\Y_{-k}\Y_{-k}'-q\bfI_{p-1})^{-1} \big) \Big]\\
&\quad - \E[Y_{k1}Y_{k2}] \E \Big[1_n' \Y_{-k}'  (\Y_{-k}\Y_{-k}'-q\bfI_{p-1})^{-1} \Y_{-k} 1_n\Big].
\end{split}
\end{equation}
Since the matrix $1_n1_n'$ has rank one, and 
$$\text{Spectrum}(\Y_{-k}'  (\Y_{-k}\Y_{-k}'-q\bfI_{p-1})^{-1} \Y_{-k})\cup \{0\} =\text{Spectrum}( (\Y_{-k}\Y_{-k}'-q\bfI_{p-1})^{-1} \Y_{-k}\Y_{-k}' )\cup \{0\},$$ we have 
\begin{align}
|1_n' &\Y_{-k}'  (\Y_{-k}\Y_{-k}'-q\bfI_{p-1})^{-1} \Y_{-k} 1_n| = | \tr\Big( \Y_{-k}' (\Y_{-k}\Y_{-k}'-q\bfI_{p-1})^{-1} \Y_{-k} 1_n1_n' \Big)| \nonumber\\
&\le  \twonorm{ (\Y_{-k}\Y_{-k}'-q\bfI_{p-1})^{-1} \Y_{-k}\Y_{-k}'}  \,\twonorm{ 1_n1_n' } = \twonorm{\bfI + q(\Y_{-k}\Y_{-k}'-q\bfI_{p-1})^{-1}}\, n \nonumber \\
&\le n \Big(1+\frac{|q|}{\Im(q)}\Big) \lesssim \frac{n^{3/2}}{p^{1/2}}\,. \label{eq:new12}
\end{align}
Next we bound the terms in \eqref{eq:1236}.
To this end, we will repeatedly use the following fact which can be found in, for example, the appendix of \cite{bai:silverstein:2010}: For a real, symmetric, positive semidefinite $p\times p$ matrix $\bfC$, $z\in \C$ with $\Im(z)>0$ it holds
\begin{equation}\label{eq:le1}
\twonorm{(\bfC-z\bfI)^{-1}}\le \frac{1}{\Im(z)}\,.
\end{equation}
Using \cite[eq.~(A.1.12)]{bai:silverstein:2010} and \eqref{eq:le1}, 
one gets for $p$ sufficiently large (and uniformly in $k$), 
\begin{align}
|\E[\vep_k]|
& \lesssim \frac{1}{n} + \frac{|q|}{n \Im(q)} + |\E[Y_{k1} Y_{k2}]|\,p 
+|\E[Y_{k1} Y_{k2}]| \,\frac{p|q|}{\Im(q)}  
 +|\E[Y_{k1} Y_{k2}]| \, \frac{n^{3/2}}{p^{1/2}} \nonumber \\
 &\lesssim \frac{1}{n} + \frac{1}{n \Im(q)} + |\E[Y_{k1} Y_{k2}]| \, p \Big(1+\frac{1}{\Im(q)}\Big) +|\E[Y_{k1} Y_{k2}]| \, \frac{n^{3/2}}{p^{1/2}} \nonumber \\
 &\lesssim 
  \frac{1}{n^{1/2} p^{1/2}}+ |\E[Y_{k1} Y_{k2}]| \, \frac{n^{3/2}}{p^{1/2}}\, 
 \,. \label{eq:sgfedsgf} 
\end{align}

By (3.3.13) and p.~57 in \cite{bai:silverstein:2010}, we have
\begin{equation}\label{eq:dsgsdss}
\begin{split}
\Im(1-q-\gn -\gn q \E[s_{\bfR}(q)])&\le - \Im (q)\,,\\
\Im(1-q-\gn -\gn q \E[s_{\bfR}(q)]+\vep_k)&\le - \Im (q)\,.
\end{split}
\end{equation}
Applying  \eqref{eq:dsgsdss} and \eqref{eq:sgfedsgf} to the term $J_1$ in \eqref{eq_def_J}, we see that
\begin{equation*}
\gn^{1/2} |J_1| \lesssim \frac{n^{1/2}}{p^{3/2}} \sum_{k=1}^p |\E[\vep_k]|\lesssim \frac{1}{p}+\frac{n^2}{p^2} \sum_{k=1}^p |\E[Y_{k1}Y_{k2}]| = o(1)\,, \quad \pto\,,
\end{equation*}
where assumption \eqref{eq_mom_cond_unit_sphere} was used for the last equality.
Another application of \eqref{eq:dsgsdss}, this time to $J_2$ defined in \eqref{eq_def_J}, shows that 
\begin{equation*}
\begin{split}
\gn^{1/2}|J_2|&\le \frac{1}{p \, \gn \, v^3} \sum_{k=1}^p \E\big[|\vep_k|^2\big]\\
&\lesssim  \frac{1}{p \, \gn } \sum_{k=1}^p 
\Big( \E\big[|\vep_k-\E_k[\vep_k]|^2\big]  +\E\big[|\E_k[\vep_k]-\E[\vep_k]|^2\big]+|\E[\vep_k]|^2\big] \Big)\\
&=: J_{21}+J_{22}+J_{23}\,,
\end{split}
\end{equation*}
where $\E_k[\cdot]$ denotes the conditional expectation given $\Y_{-k}$.

In view of \eqref{eq:sgfedsgf} and assumption \eqref{eq_mom_cond_unit_sphere}, $J_{23}$ converges to zero. Next, we investigate the term $J_{21}.$ 
Writing $\bfB_k=(b_{ij})= \Y_{-k}' (\Y_{-k}\Y_{-k}'-q\bfI_{p-1})^{-1} \Y_{-k}$ and using \eqref{mean_yty} for the third equality, we have 
\begin{equation*}
\begin{split}
-\vep_k+\E_k[\vep_k]&= \y_k \bfB_k \y_k'- \E_k[\y_k \bfB_k \y_k']\\
&=\y_k \bfB_k \y_k' - \tr( \bfB_k \E[\y_k' \y_k])\\
&= \y_k \bfB_k \y_k' -\frac{1}{n}\tr( \bfB_k)+\E[Y_{k1} Y_{k2} ] (\tr( \bfB_k)-1_n'\bfB_k 1_n)
\end{split}
\end{equation*}
and therefore,
\begin{equation} \label{eq:defA}
\begin{split}
\E_k\Big[ \big| \vep_k-\E_k[\vep_k]\big|^2 \Big]
&\lesssim
\, \E_k\Big[ \Big| \y_k \bfB_k \y_k' -\frac{1}{n}\tr( \bfB_k) \Big|^2 \Big] +\, |\E[Y_{k1} Y_{k2}]|^2(|\tr(\bfB_k)|^2 + |1_n'\bfB_k 1_n|^2)\\
&=: A_{1k}+A_{2k}\,.
\end{split}
\end{equation}
We continue by analyzing the term $A_{2k}$. Using \eqref{eq:le1} we get
\begin{equation}\label{eq:trB}
|\tr( \bfB_k)|= \Big|q \tr\big((\Y_{-k}\Y_{-k}'-q\bfI_{p-1})^{-1} \big) + \tr(\bfI_{p-1})\Big| \le |q| \frac{p}{\Im(q)}+p\lesssim  \sqrt{np}\,
\end{equation}
and \eqref{eq:new12} yields
\begin{equation}\label{eq:trB1}
\E[|1_n'\bfB_k 1_n|^2] \lesssim  \frac{n^3}{p}\,.
\end{equation}
Then a combination of \eqref{eq:trB} and \eqref{eq:trB1} yields
\begin{equation}\label{eq:A2}
\frac{n}{p^2} \sum_{k=1}^p \E[A_{2k}]\lesssim  \frac{n^4}{p^3}  \sum_{k=1}^p |\E[Y_{k1}Y_{k2}]|^2 \le \bigg( \frac{n^2}{p^{3/2}}\sum_{k=1}^p |\E[Y_{k1}Y_{k2}]| \bigg)^2  \to 0\,,\qquad \pto\,,
\end{equation}
where assumption \eqref{eq_mom_cond_unit_sphere} was used for the last step.
Regarding $A_{1k}$, one has 
\begin{equation*}
A_{1k} \lesssim 
 \, \E_k\bigg[ \Big|\sum_{1\le i\neq j\le n} b_{ij}  Y_{ki}Y_{kj}\Big|^2 \bigg]  + \, \E_k\bigg[ \Big|\sum_{i=1}^n b_{ii}  (Y_{ki}^2-n^{-1})\Big|^2 \bigg]
=: A_{1k1}+A_{1k2}\,.
\end{equation*} 

Using $\Y_{-k}\Y_{-k}'=(\Y_{-k}\Y_{-k}'-q\bfI_{p-1})+q \bfI_{p-1}$, \eqref{eq:le1} and $|q|\to 1$, we also have 
\begin{equation*}
\begin{split}
\tr(\bfB_k\bfB_k^*)&= \tr\big( (\Y_{-k}\Y_{-k}'-q\bfI_{p-1})^{-1} \Y_{-k}\Y_{-k}'  (\Y_{-k}\Y_{-k}'-\overline{q} \bfI_{p-1})^{-1} \Y_{-k}\Y_{-k}'   \big)\\
&= \tr \big( (\bfI_{p-1} +q (\Y_{-k}\Y_{-k}'-q\bfI_{p-1})^{-1})  (\bfI_{p-1} +\overline{q} (\Y_{-k}\Y_{-k}'-\overline{q}\bfI_{p-1})^{-1})\big)\\
&= \tr(\bfI_{p-1}) +q \tr\big( (\Y_{-k}\Y_{-k}'-q\bfI_{p-1})^{-1} \big) +\overline{q} \tr\big( (\Y_{-k}\Y_{-k}'-\overline{q}\bfI_{p-1})^{-1}) \big)+\\
&\qquad |q|^2 \tr\big( (\Y_{-k}\Y_{-k}'-q\bfI_{p-1})^{-1} (\Y_{-k}\Y_{-k}'-\overline{q}\bfI_{p-1})^{-1})\big)\\
&\le p+ 2 |q|\, p \,\twonorm{(\Y_{-k}\Y_{-k}'-q\bfI_{p-1})^{-1}} +|q|^2 p \twonorm{(\Y_{-k}\Y_{-k}'-q\bfI_{p-1})^{-1}}^2\\
&\lesssim p +\frac{p}{\Im(q)}+\frac{p}{(\Im(q))^2}\lesssim  n. 
\end{split}
\end{equation*}
Then after expanding the expression for $A_{1k1}$, we conclude from the above inequality  
that 
\begin{equation}\label{eq:agesags}
\E[A_{1k1}]\lesssim  \E[Y_{k1}^2Y_{k2}^2] \E \bigg[\sum_{i, j=1}^n |b_{ij}|^2\bigg] \lesssim \frac{1}{n^2} \E[\tr(\bfB_k\bfB_k^*)]\lesssim  \frac{1}{n}.
\end{equation}
Recall that $\E[Y_{k1}^2]=n^{-1}$. For $A_{1k2}$, we have, as $\pto$,
\begin{equation*}
\begin{split}
\E[A_{1k2}]& \leq \sum_{1\le i\neq j\le n} \E[b_{ii}\overline{b}_{jj}] (\E[Y_{ki}^2 Y_{kj}^2]-n^{-2})+  \sum_{i=1}^n \E[|b_{ii}|^2] \E[Y_{ki}^4]\\
&\lesssim (\E[Y_{k1}^4] n^{-1} \E[|\tr(\bfB_k)|^2]+\E[Y_{k1}^4] \E[\tr(\bfB_k\bfB_k^*)] )\,,
\end{split}
\end{equation*}
where we used 
\begin{equation*}
|\E[Y_{k1}^2 Y_{k2}^2] -n^{-2}| = \Big| \frac{1}{n^2 (n-1)} - \frac{\E[Y_{k1}^4]}{n-1}\Big| \lesssim  \E[Y_{k1}^4] n^{-1} \,.
\end{equation*}
Therefore, one deduces using \eqref{eq_mom_cond_unit_sphere}
\begin{equation*}
\frac{n}{p^2} \sum_{k=1}^p \E[A_{1k}] \lesssim    \frac{n}{p^2} \sum_{k=1}^p \big( \E[A_{1k1}]+\E[A_{1k2}]\big)\lesssim
\frac{1}{p} +\frac{n^2}{p^2} \sum_{k=1}^p \E[Y_{k1}^4] \to 0\,,\quad \pto\,.
\end{equation*}
Combining this with \eqref{eq:A2} and \eqref{eq:defA} shows that
\begin{equation*}
J_{21}\lesssim \frac{n}{p^2} \sum_{k=1}^p \E[\E_k\big[|\vep_k-\E_k[\vep_k]|^2\big]]\le \frac{n}{p^2} \sum_{k=1}^p (\E[A_{1k}]+\E[A_{2k}])\to 0\,, \quad \pto\,.
\end{equation*}
Regarding $J_{22}$, we have
\begin{equation*}
\E_k [\vep_k] -\E[\vep_k]= \Big(\frac{1}{n} -\E[Y_{k1} Y_{k2} ]\Big) \big(\tr(\bfB_k) -\E[\tr(\bfB_k)]\big) + \E[Y_{k1} Y_{k2}] 1_n' (\bfB_k-\E[\bfB_k]) 1_n\,.
\end{equation*}
For the same arguments as \eqref{eq:trB1}, we obtain that 
\begin{align}\label{eq:trB112}
\frac{n}{p^2} \sum_{k=1}^p |\E[Y_{k1} Y_{k2} ]|^2 \E[|1_n'(\bfB_k-\E[\bfB_k]) 1_n|^2] &\lesssim \frac{n^4}{p^3} \sum_{k=1}^p |\E[Y_{k1} Y_{k2} ]|^2 \nonumber \\
&\le \bigg( \frac{n^2}{p^{3/2}}\sum_{k=1}^p |\E[Y_{k1}Y_{k2}]| \bigg)^2 \to 0\,, \qquad \pto.
\end{align}
Furthermore, we see that
\begin{equation*}
\begin{split}
\frac{1}{n^2} \E\big[| \tr( \bfB_k)- \E[\tr(\bfB_k)]|^2\big]
&= \frac{|q|^2}{n^2} \E\big[ \big| \tr((\Y_{-k}\Y_{-k}'-q\bfI_{p-1})^{-1})- \E[\tr((\Y_{-k}\Y_{-k}'-q\bfI_{p-1})^{-1})]   \big|^2 \big]\\
&= |q|^2 \E\Big[ \Big| s_{\Y_{-k}'\Y_{-k}}(q)- \E\big[s_{\Y_{-k}'\Y_{-k}}(q)\big] \Big|^2 \Big]\,,
\end{split}
\end{equation*}
where \eqref{eq:sgdsglop} was used for the last identity.
As in the proof of Lemma~\ref{lem:expectedtransform}, 
let $\mathcal{F}_j$ denote the filtration generated by $\y_i,1\le i\le j, i\neq k$, and we define $\mathcal{F}_ 0 = \{ \emptyset \}$.
Then one can show analogously to \eqref{eq:martdiff} that
\begin{equation}\label{eq:martdiff1}
s_{\Y_{-k}'\Y_{-k}}(q)- \E\big[s_{\Y_{-k}'\Y_{-k}}(q)\big]= \sum_{j=1;j\neq k}^p \left(  \E[s_{\Y_{-k}'\Y_{-k}}(q)| \mathcal{F}_j]-\E[s_{\Y_{-k}'\Y_{-k}}(q)|  \mathcal{F}_{j-1}] \right)
=:\sum_{j=1;j\neq k}^p W_{j,k}\,
\end{equation}
is a sum of bounded martingale differences. Indeed, with the same arguments as in the proof of Lemma~\ref{lem:expectedtransform} we obtain $|W_{j,k}|\lesssim  n^{-1} \gn^{-1/2}$. An application of the extended Burkholder inequality (Lemma 2.12 in \cite{bai:silverstein:2010}) yields
\begin{equation}\label{eq:dfsdfexawkl}
\E\bigg[ \Big| \sum_{j=1;j\neq k}^p W_{j,k}\Big|^2 \bigg]\lesssim \,  \E\bigg[ \sum_{j=1 ; j\neq k}^p |W_{j,k}|^2 \bigg] \lesssim \frac{1}{n}\,,
\end{equation}
which combined with \eqref{eq:trB112} shows that $J_{22} \to 0$. 
The proof of \eqref{eq:grsgrsd} is complete.
\end{proof}

Next, we turn to the proof of Theorem \ref{thm_unit_sphere} part (2), where it is assumed that $p/n\to \gamma>0$.

\begin{lemma} \label{lem_mp}
    Assume $\gn \to \gamma > 0$ and let $z\in\mathbb{C}^+$. Then, as $\pto$,  it holds
    $s_{\bfR}(z) \to s_{F_{\gamma}}(z)$ almost surely, where the Stieltjes transform $s_{F_{\gamma}}$ of $F_{\gamma}$ is defined in \eqref{eq_stieltjes_mp}.
\end{lemma}

\begin{proof}
Let $z\in \mathbb{C}^+.$
Using Burkholder's inequality, it is straightforward to show that 
$s_{\bfR} (z) - \E [ s_{\bfR} (z) ] \to 0$ almost surely (see, e.g., \cite[Proof of Theorem 1.1, Step 1]{bai:zhou:2008}). Thus, it is left to investigate the asymptotic behavior of the non-random part $\E[s_{\bfR}(z)].$ 
    For this purpose, we will show that 
\begin{align} \label{aim_mp}
    \lim_{\pto} \E\big[s_{\bfR}(1+\gamma^{1/2} z)\big] 
    = s_{F_\gamma} ( 1 + \gamma ^{1/2} z), \quad z\in \mathbb{C}^+.
\end{align}
Consider \eqref{eq:sols1}. Whenever $\delta_p \to 0,$ we have that $\E[s_{\bfR}( 1 + \gamma ^{1/2} z)]$ converges to $s_{F_{\gamma}}(1+\gamma^{1/2}z).$ In this case, the Lipschitz continuity of $s_{\bfR}$ implies \eqref{aim_mp}. Note that in the case $\gamma>0$, the proof of assertion $\delta_p \to 0$ is identical to the proof of \eqref{eq:grsgrsd}, which establishes the desired result. 
\end{proof}

\subsection{Proof of Theorem~\ref{thm:spearman}} \label{sec_proof_rho}
In this section, we provide a proof for the limiting spectral distribution of Spearman's rho matrix $\SR$ given in Theorem \ref{thm:spearman}. The proof is an application of Theorem \ref{thm_unit_sphere} and relies on verifying the moment conditions in \eqref{eq_mom_cond_unit_sphere}.  To the best of our knowledge, such moments have not been considered in the literature for non-continuous entry distributions of $\X$ and our more general definition of $\bfZ$. It is worth mentioning that the results of \cite[p.~429]{bai:zhou:2008} apply in the continuous case only. 

\begin{proof} 
We want to verify the conditions of Theorem \ref{thm_unit_sphere}. 
Using $\sum_{j=1}^n Z_{ij} = 0$ and $\sum_{j=1}^n Z_{ij}^2 = 1$ for $1 \leq i \leq p$, it is easy to check that $\E[Z_{i1}]=0$ and
\begin{align*}
 \frac{n^2}{p^{3/2}} \sum_{i=1}^p \E [ Z_{i1} Z_{i2} ]=\frac{n^2}{p^{3/2}}  \cdot \frac{-p}{n(n-1)} =o(1)\,, \qquad \pto\,.
\end{align*}
The difficult part is to establish the analogue to the first condition in \eqref{eq_mom_cond_unit_sphere}, that is,
\begin{align} \label{4thmoment_z}
\frac{n^2}{p^2} \sum_{i=1}^p \E[Z_{i1}^4] =o(1)\,, \qquad \pto\,,
\end{align}
which will be proven in Subsection~\ref{lem:difficult} below.
Therefore, the assertion of Theorem \ref{thm:spearman} follows from Theorem \ref{thm_unit_sphere}. 
\end{proof}

\subsubsection{Proof of \eqref{4thmoment_z}} \label{lem:difficult}
Set $M_{ij} := Q_{ij} - (n+1)/2$ for $1 \leq j \leq n, 1 \leq i \leq p.$ Since $Z_{i1}=M_{i1}/\sqrt{\sum_{j=1}^n M_{ij}^2}$, we start by analyzing $\sum_{j=1}^n M_{ij}^2$.
If $X_{i1}$ follows a continuous distribution, we get $\{Q_{i1},\ldots, Q_{in}\}=\{1,\ldots, n\}$ \as\ and as a consequence
\begin{align*}
    \sum_{j=1}^n M_{ij}^2 = \frac{n (n^2 -1)}{12} 
    \sim \frac{n^3}{12} \quad \as 
\end{align*}
Our aim is to show that, for all non-degenerate distributions of $X_{i1}$, $\sum_{j=1}^n M_{ij}^2$ is at least of order $n^3$ with high probability.

By Assumption \ref{ass_asympt_non_degen}, there exists some small $\eta>0$ such that for all real-valued sequences
$(t_{i,p})_{p\in\N}$ 
\begin{align}\label{eq_asympt_non_degen11}
    \limsup_{p\to\infty} \max_{1\le i\le p} \P \lb X_{i1} = t_{i,p} \rb \le 1-\eta < 1. 
\end{align}
Let $1 \leq i \leq p$. For ease of presentation, we assume\footnote{Otherwise, all arguments are valid in the limit $\pto$.} $\P ( X_{i1} = t ) \le 1-\eta$ for all $t\in \R$. This implies $\P(X_{i1}<t)+\P(X_{i1}>t)\ge \eta$ all $t\in \R$. We deduce that there exists a $t_i\in \R$ such that at least one of (i) $\min\big( \P(X_{i1}\le t_i),\P(X_{i1}>t_i)\big)\ge \eta/2$, or (ii) $\min\big( \P(X_{i1}< t_i),\P(X_{i1}\ge t_i)\big)\ge \eta/2$ must hold.  We restrict ourselves to (i), as the arguments under (ii) are almost identical.

The following result shows that $T_i:=\sum_{j=1}^n M_{ij}^2$ decreases if we add a tied observation.
\begin{lemma}\label{lem_rank}
 Let $1 \leq i \leq p$.   Assume $|\{ X_{i1}, \ldots, X_{in} \} | = k \geq 3$ and write $\{ X_{i1}, \ldots, X_{in} \} = \{ V_{i1}, \ldots, V_{ik}\}$, where $V_{i1}<\cdots<V_{ik}$. For $l=2, \ldots, k-1$, we write $T_{l,+}$ and $T_{l,-}$ for a modification of $T_i$, where all $X_{ij}$ taking on the value $V_{il}$ are replaced by $V_{i,l+1} $ and $V_{i,l-1}$, respectively. Then, it holds $T_{l,+} \leq T_i$ and $T_{l,-}  \leq T_i$. 
\end{lemma}
For the sake of brevity, we omit the proof of Lemma \ref{lem_rank}, which is a simple application of the following consequence of Hölder's inequality: Let $a_1, \ldots, a_m \in \R ,$ then it holds
\begin{align*}
    m \lb \frac{1}{m} \sum_{i=1}^m a_i \rb^2 \leq \sum\limits_{i=1}^m a_i^2. 
\end{align*}
We turn back to the proof of \eqref{4thmoment_z}.
We apply  Lemma \ref{lem_rank} multiple times, in such a way that finally all (original) $X_{ij}\le t_i$ are set to the (original) $V_{i1}$, and all (original) $X_{ij}> t_i$ are set to the (original) $V_{ik}$.
Therefore, we may restrict ourselves to the case $k=2$ and assume $V_{i1}<V_{i2}$. (Note that $k=1$ constitutes the degenerated case.)  
Let $m_{il} =   | \{ j : X_{ij} = V_{il} \} | $ 
for $l \in \{1,2\}$ and let $r_{il}$ be the rank of observations taking on the value $V_{il}.$  
Note that 
\begin{align*}
    r_{i1} = \frac{1}{m_{i1}} \sum\limits_{a=1}^{m_{i1}} a  = \frac{m_{i1} + 1}{2}, \quad 
    r_{i2} = \frac{1}{m_{i2}} \sum_{a=m_{i1} + 1}^{m_{i1} + m_{i2}} a = m_{i1} + \frac{m_{i2} + 1}{2}, 
\end{align*}
which implies 
\begin{align*} 
    T_i & \geq  m_{i1} \lb r_{i1} - \frac{n + 1}{2} \rb ^2
    + m_{i2} \lb r_{i2} - \frac{n + 1}{2} \rb ^2 
    = m_{i1} \lb \frac{m_{i1} - n}{2} \rb^2 
    + m_{i2} \lb \frac{m_{i1}}{2} \rb^2 
		=\frac{1}{4} n \, m_{i1}\, m_{i2}.
\end{align*}

Choose $\varepsilon=\eta/4$ and consider the event $A_n=\{\min(m_{i1}, m_{i2})\ge n\varepsilon\}$. In fact, this event occurs with high probability, which we will show in the following.
By construction, we have $m_{i1}\sim \text{Bin}(n,q_i)$ and $m_{i2}=n-m_{i1}\sim \text{Bin}(n,1-q_i)$, where $q_i=\P(X_{i1}\le t_i)$. 
We then get
\begin{align*}
   \P(A_n^c)= \P \lb \min(m_{i1}, m_{i2}) < n \varepsilon \rb
    \leq \P \lb m_{i1} < n \varepsilon \rb 
    + \P \lb m_{i2} < n \varepsilon \rb,
\end{align*}
and by Hoeffding's inequality 
\begin{align}  
    \P\lb m_{i1} < n \varepsilon \rb
    &= \P  \lb - \sum\limits_{j=1}^n (\1(X_{ij}\le t_i) - q_i) > n (q_i - \varepsilon) \rb \notag\\
		&\leq \exp \lb - 2 n ( q_i - \varepsilon)^2 \rb \le \exp \lb - 2 n \,(\eta/4)^2 \rb\,, \label{prob_hoeffding}
\end{align}
where $q_i\ge \eta/2$ was used for the last step.
A similar bound can be derived for $\P \lb m_{i2} < n \varepsilon \rb$. 
Thus, we have 
\begin{align*}
    \E [ Z_{i1}^4 ] 
    & = 
    \E \left[ Z_{i1}^ 4 \1_{A_n} \right] 
    + \E \left[ Z_{i1}^ 4 \1_{A_n^c} \right]
     \leq
    \E \left[ \lb \frac{M_{i1}}{\sqrt{T_i}}\rb^ 4 \1_{A_n} \right]
    + \P
    \lb A_n^c \rb
    = O(n^{-2}),
\end{align*}
where we used $|M_{i1}|\leq (n-1)/2$ and the fact that $T_i\ge \frac{\vep^2}{4} n^3$ on $A_n$ for the first summand, and \eqref{prob_hoeffding} and Assumption \ref{ass_asympt_non_degen} for the second summand. This implies \eqref{4thmoment_z}.

\section{Proof of Theorem~\ref{thm:kendall}} \label{sec_proof_kendall} 
This section is devoted to the proof of Theorem~\ref{thm:kendall}, which is divided into several steps. In Section~\ref{sec_structure_kendall_proof}, we outline the structure of the proof based on the preparations conducted in Sections~\ref{sec_proof_kendall_sec1} and \ref{sec_proof_kendall_sec2}.
The main claims are proven in Sections~\ref{sec_proof_kendall_claim1}, \ref{sec_proof_kendall_claim2} and \ref{sec_proof_kendall_claim3}.

\subsection{Some useful notations} \label{sec_proof_kendall_sec1}
W.l.o.g. we may that assume that the random variables $X_{it}$ are centered and standardized for $1 \leq i \leq p, ~ 1 \leq t \leq n$. 
First, we introduce some notations. Setting 
\begin{align*}
    \bfM & := \frac{2}{n (n -1) }   \sum_{1\le s<t\le n} \offdiag\lb \sign(\mathbf{q}_s-\mathbf{q}_t) (\sign(\mathbf{q}_s-\mathbf{q}_t))' \rb \\
   &  = 
    \frac{2}{n (n -1) }   \sum_{1\le s<t\le n} \offdiag\lb \sign(\x_s-\x_t) (\sign(\x_s-\x_t))' \rb\,,
\end{align*}
we have
\begin{align} \label{rep_t_m}
    \bfT = \bfD^{-1/2} \bfM  \bfD^{-1/2}  .
\end{align}
It will be crucial to approximate the matrix $\bfD$ by $\tilde\bfD, $ which is formally justified below. To this end, we rewrite the entries of $\bfD$ as 
\begin{align}
   D_{ii} & = \frac{3}{n}\sum_{s=1}^n \hat U_{is}^2 , \quad \text{ where }
   \hat U_{is}  = 2\hat{H}_i(X_{is}) - 1 - \hat q_i(X_{is}), \label{eq_formula_D}\\ 
   \hat H_i(x) & = \frac{1}{n} \sum_{k=1}^n \1 \{ X_{ik} \leq x \}, \quad
   \hat q_i(x)  = \frac{1}{n} \sum_{k=1}^n \1 \{ X_{ik} = x \}, \nonumber
\end{align}
for $1 \leq i \leq p, 1 \leq s \leq n$, $x\in\R$. 
Furthermore, for $1 \leq i \leq p$, let $H_i$ be the c.d.f.~of $X_{i1}$ and $\P (X_{i1}=x) = q_i(x)$, $x\in\R$. Now we define $\tilde\bfD$ as the $p\times p$ diagonal matrix with diagonal elements 
  \begin{align} \label{def_tildeD}
        \tilde D_{ii} = \frac{3}{n} \sum_{s=1}^n \lb 2 H_i(X_{is})-1 - q_i(X_{is}) \rb ^2, \qquad 1 \leq i \leq p,
    \end{align}

\subsection{Analysis of $\bfM$} \label{sec_proof_kendall_sec2}

 For $1\le s,t\le n$ and $s\neq t$, let $\bfB_{(s,t)}=\sign(\x_s-\x_t)$ and 
consider the decomposition
\begin{equation}\label{eq:decA}
\bfB_{(s,t)}=\bfA_{(s,t)}+ \bfA_{(s,\cdot)}+ \bfA_{(\cdot,t)}\,,
\end{equation}
where 
$\bfA_{(s,\cdot)}= \E[\bfB_{(s,t)}|\x_s]\,,   \bfA_{(\cdot,t)}=\E[\bfB_{(s,t)}|\x_t]$ and $\bfA_{(s,t)}=\bfB_{(s,t)}-\bfA_{(s,\cdot)}- \bfA_{(\cdot,t)}$.
Observe that 
\begin{equation} \label{def_B}
\bfM = \frac{2}{n(n-1) }   \sum_{1\le s<t\le n} \offdiag \lb \bfB_{(s,t)} (\bfB_{(s,t)})' \rb \,.
\end{equation}
Next, we decompose $\offdiag(\bfB_{(s,t)} (\bfB_{(s,t)})')$.
Using \eqref{eq:decA}, we have 
\begin{equation*}
\offdiag(\bfB_{(s,t)} (\bfB_{(s,t)})')  = \bfM_{(s,t)}^{(1)}+\bfM_{(s,t)}^{(2)}+(\bfM_{(s,t)}^{(2)})'+\bfM_{(s,t)}^{(3)}\,,
\end{equation*}
where
\begin{equation*}
\begin{split}
\bfM_{(s,t)}^{(1)}&= \offdiag((\bfA_{(s,\cdot)}+\bfA_{(\cdot, t)}) (\bfA_{(s,\cdot)}+\bfA_{(\cdot, t)})') , \\
\bfM_{(s,t)}^{(2)}&= \offdiag( \bfA_{(s,t)} (\bfA_{(s,\cdot)}+\bfA_{(\cdot, t)})'),  \\
\bfM_{(s,t)}^{(3)}&= \offdiag( \bfA_{(s,t)} (\bfA_{(s,t)})')\,.
\end{split}
\end{equation*}
Thus combining \eqref{def_B} with the definition
\begin{equation*}
\bfM^{(k)} := \frac{2}{n (n-1)} \sum_{1\le s<t\le n} \bfM_{(s,t)}^{(k)}  \,, \qquad k=1,2,3\,,
\end{equation*}
we have $\bfM= \bfM^{(1)}+\bfM^{(2)}+(\bfM^{(2)})'+\bfM^{(3)}$.

Next, we write $\bfM^{(1)}$ as 
\begin{equation} \label{m1_rep}
\begin{split}
\bfM^{(1)} &= \frac{2}{n(n-1) }  \sum_{1\le s<t\le n} \offdiag( (\bfu_s-\bfu_t)(\bfu_s-\bfu_t)')  , \\
\end{split}
\end{equation}
where we used the notation 
\begin{equation} 
\bfu_s:=(U_{1s},\ldots, U_{ps})':=\bfA_{(s,\cdot)}=-\bfA_{(\cdot,s)}\,\qquad 1\le s\le n\,.
\end{equation} 
Here, the last equality follows directly from the definition. In particular, for $1\le i\le p$ and $t\neq s$, we have 
\begin{equation} \label{x1}
\begin{split}
U_{is} = \E[\sign(X_{is}-X_{it})| \x_s]&= \P(X_{is}>X_{it}| \x_s)-\P(X_{is}<X_{it}| \x_s)\\
&= 2 H_i(X_{is})-1 - q_i(X_{is}).
\end{split}
\end{equation}
Note that $\bfu_1, \ldots, \bfu_n$ are i.i.d. with independent and centered components. 
Using \eqref{m1_rep}, we may decompose $\bfM^{(1)}$ further as 
\begin{align}
    \bfM^{(1)} 
     &= \frac{2}{n }  \sum_{s=1}^n \ \offdiag( \bfu_s \bfu_s') 
     - \frac{2}{n(n-1)} \sum_{1\le s\neq t\le n}  \offdiag( \bfu_s \bfu_t') 
      \nonumber \\
    & = \frac{2}{n  } \sum\limits_{s=1}^n  \bfu_s \bfu_s' 
    - \frac{2}{n } \sum\limits_{s=1}^n \diag \lb \bfu_s \bfu_s' \rb 
    - \frac{2}{n(n-1) } \sum_{1\le s\neq t\le n}  \offdiag( \bfu_s \bfu_t')  \nonumber \\ 
    & = \bfM^{(1,1)} + \bfM^{(4)},
    \label{decomp_M1}
\end{align}
where 
\begin{align*}
    \bfM^{(1,1)} 
    & := \frac{2}{n } \sum\limits_{s=1}^n  \bfu_s \bfu_s'- \frac{2}{n  } \sum\limits_{s=1}^n  \diag \lb \bfu_s \bfu_s' \rb=
     \frac{2}{n } \sum\limits_{s=1}^n  \bfu_s \bfu_s'- \frac{2}{3}  \tilde\bfD , \\ 
    \bfM^{(4)}  &:= - \frac{2}{n(n-1) } \sum_{1\le s\neq t\le n} \offdiag( \bfu_s \bfu_t')\,.  
\end{align*}
Altogether, we have derived the following decomposition of $\bfM$, 
\begin{align}  \label{decomp_Ttilde}
\bfM= \bfM^{(1,1)}+\bfM^{(2)}+(\bfM^{(2)})'+\bfM^{(3)}+ \bfM^{(4)}.
\end{align}

\subsection{Structure of the proof} \label{sec_structure_kendall_proof}
The proof of Theorem~\ref{thm:kendall} now consists of the following three claims, which will be shown in the remainder of this section.
\begin{itemize}
\item[$(i)$] 
If $p/n\to \gamma\in [0,\infty)$, the LSDs of $\sqrt{n/p}\,\bfT$ and $\sqrt{n/p}\,\tilde \bfD^{-1/2} \bfM  \tilde \bfD^{-1/2}$ coincide. 
\item[$(ii)$] 
If $p/n\to \gamma\in [0,\infty)$, the LSDs of $\sqrt{n/p}\,\tilde \bfD^{-1/2} \bfM  \tilde \bfD^{-1/2}$ and $\sqrt{n/p}\,\tilde \bfD^{-1/2} \bfM^{(1,1)}  \tilde \bfD^{-1/2}$ coincide. 
\item[$(iii)$] If $p/n\to \gamma\in (0,\infty)$, the ESDs of $\tilde \bfD^{-1/2} \bfM^{(1,1)}  \tilde \bfD^{-1/2}$ converge in probability to the distribution of $\tfrac{2}{3} (\eta-1)$, where $\eta$ follows the \MP law with parameter $\gamma$. \\
If $p/n\to 0$, the ESDs of $\sqrt{n/p}\,\tilde \bfD^{-1/2} \bfM^{(1,1)}  \tilde \bfD^{-1/2}$ converge in probability to the distribution of $\frac{2}{3}\zeta$, where $\zeta$ follows the semicircle law. 
\end{itemize}

\subsection{Proof of claim $(i)$} \label{sec_proof_kendall_claim1}
To prove claim $(i),$ we first give some auxiliary results. The next lemma regards an approximation of $\bfD$  in \eqref{eq_def_D} by the diagonal matrix $\tilde\bfD$ in \eqref{def_tildeD}. 
\begin{lemma}
    \label{lem_dhat}
    It holds 
    \begin{align*}
      \frac{n}{p^2} \| \bfD - \tilde{\bfD} \|_F^2 = o_{\mathbb{P}}(1), \qquad p\to\infty. 
    \end{align*}
\end{lemma}
\begin{proof}[Proof of Lemma \ref{lem_dhat}]
  As a preparation, we get from the Dvoretzky–Kiefer–Wolfowitz inequality (see \cite[Corollary 1 and Comment 2 (iii)]{massart:1990}) that for every $1 \leq i \leq p$ and $\vep>0$
    \begin{align}
        \P \lb \left| U_{i1} - \hat U_{i1} \right| > \varepsilon  \rb 
        & \leq \P \lb \sup_{x\in \R} \left| 2 \hat H_i(x) - \hat q_i(x) - \lb 2 H_i(x) - q_i(x) \rb \right| > \varepsilon \rb \nonumber \\ 
        & \leq \P \lb \sup_{x\in \R} \left| 2 \hat H_i(x) - 2 H_i(x) \right| > \frac{ \varepsilon }{2} \rb
        +\P \lb \sup_{x\in \R} \left| \hat q_i(x) - q_i(x)  \right| > \frac{ \varepsilon }{2} \rb \nonumber \\ 
        & \leq \P \lb \sup_{x\in \R} \left|  \hat H_i(x) -  H_i(x) \right| > \frac{ \varepsilon }{4} \rb
        +2 \P \lb \sup_{x\in \R} \left| \hat H_i(x) - H_i(x)  \right| > \frac{ \varepsilon }{4} \rb \nonumber \\  
        & \leq 
         6 \exp \lb - \frac{2 n \varepsilon^2}{16 } \rb \,. \label{eq_dkw}
    \end{align}
  This implies, for some $c>0,$
    \begin{align*}
\E\left[\left(\hat U_{i1}-U_{i1}\right)^2\right]
&=
2\int_0^\infty
t \, \P\left(\left|\hat U_{i1}-U_{i1}\right|>t\right)\,dt \leq
12\int_0^\infty t \exp\left(-c n t^2\right)\,dt 
\\ & =  \frac{12}{2cn}\int_0^\infty e^{-u}\,du
\lesssim \frac{1}{n},
\end{align*} 
where we used the substitution $u=cnt^2$. Since $U_{is}$ and $\hat U_{is}$ are uniformly bounded, we also have
\begin{align} \label{q1}
\E \left(U_{is}^2-\hat U_{is}^2\right)^2
\lesssim
\E \left(U_{is}-\hat U_{is}\right)^2 \lesssim \frac{1}{n}.
\end{align}

Combining \eqref{q1} with Jensen's inequality gives
		\begin{align*}
       \E \left[ \left\| \bfD - \tilde\bfD \right\|_F^2 \right] 
      &  = \sum_{i=1}^p \E \left[ \lb \tilde D_{ii} - D_{ii} \rb^2  \right]
       = \sum_{i=1}^p \E \left[ 
      \left\{ \frac{1}{n} \sum_{s=1}^n \lb 3 U_{is}^2 - 3 \hat U_{is}^2 \rb 
      \right\} ^2 \right]  \\
		 &\leq \frac{9}{n} \sum_{i=1}^p \sum_{s=1}^n \E\lb U_{is}^2 - \hat U_{is}^2 \rb^2 \lesssim \frac{p}{n} .
  \end{align*} 
   Thus, we have 
    $$
    \frac{n}{p^2}  \E \left[ \left\| \bfD - \tilde\bfD \right\|_F^2 \right] \lesssim \frac{n}{p^2} \frac{p}{n} = \frac{1}{p} = o(1)\,, \qquad \pto\,,
    $$ 
	and the desired statement follows by Markov's inequality.
\end{proof}

\begin{lemma} \label{lem_d_ii}
   As $\pto$, it holds with probability converging to $1$ that
    \begin{align}
        \min_{1 \leq i \leq p} D_{ii} & > C, \quad  \label{d_ii} \\ 
        \min_{1 \leq i \leq p} \tilde D_{ii} & >  C, \label{tilde_d_ii}
    \end{align}
    for some constant $C>0.$
\end{lemma}
\begin{proof}[Proof of Lemma \ref{lem_d_ii}]
First, note that \eqref{d_ii} follows along the lines of the proof of  \eqref{4thmoment_z} in Section~\ref{lem:difficult}. The constant $C$ will depend on $\eta$ in \eqref{eq_asympt_non_degen11} whose existence is guaranteed by Assumption~\ref{ass_asympt_non_degen}. For brevity, we omit details.

Next, we turn to the proof of statement \eqref{tilde_d_ii} concerning $\tilde D_{ii}$. Using the union bound and \eqref{eq_dkw}, it follows for $\varepsilon>0$ that 
\begin{align*}
     \P \lb \max_{\substack{1 \leq i \leq p, \\ 1 \leq s \leq n}} \left| U_{is} - \hat U_{is} \right| > \varepsilon  \rb 
     \leq n \sum_{i=1}^p  \P \lb \left| U_{i1} - \hat U_{i1} \right| > \varepsilon  \rb  = o(1).
\end{align*}
This implies
\begin{align*}
    \max_{1\leq i \leq p} | D_{ii} - \tilde D_{ii} | 
    \lesssim \max_{\substack{1 \leq i \leq p, \\ 1 \leq s \leq n}} \left| U_{is} - \hat U_{is} \right| = o_{\P}(1). 
\end{align*}
Combining this result with \eqref{d_ii}, the proof of \eqref{tilde_d_ii} concludes. 
\end{proof}

With these preparations in mind, we continue with the proof of claim $(i).$ That is, we have to show that the LSDs of $\sqrt{n/p}~\,\tilde \bfD^{-1/2} \bfM  \tilde \bfD^{-1/2}$ and $\sqrt{n/p}~\bfT$ coincide. 
Note that, by \cite[Corollary A.41]{bai:silverstein:2010}, the \Levy distance between the corresponding ESDs converges to zero (in probability) if $1/p$ times the squared Frobenius norm $\| \cdot \|_F$ of the difference between these to matrices tends to zero (in probability). For this purpose, we write $\bfM = \bfM^{(1,1)} + \tilde\bfM^{(2)} $ using \eqref{decomp_Ttilde} and \eqref{decomp_M1}, where 
\begin{align*}
    \tilde\bfM^{(2)} & = \bfM^{(2)}+(\bfM^{(2)})'+\bfM^{(3)} + \bfM^{(4)}.
\end{align*}
These matrices satisfy 
\begin{align} \label{f1}
    \| \bfM^{(1,1)} \| & \lesssim 1 \textnormal{ with probability converging to } 1,  
    \quad 
   \frac{n}{p^2} \| \tilde\bfM^{(2)} \|_F^2  = o_{\P}(1),
\end{align}
where the first statement follows from an application of \cite[Theorem 5.9]{bai:silverstein:2010} (after normalization of $\bfu_s$), and the second assertion can be shown similarly to \eqref{neg_terms} below.
Using Lemma~\ref{lem_dhat}, Lemma~\ref{lem_d_ii} and \eqref{f1}, we get
\begin{align*}
   &  \frac{n}{p^2} \left\| \tilde\bfD^{-1/2} \bfM \lb \bfD^{-1}\!\!  - \tilde\bfD^{-1} \rb  \right\|_F^2  
    \lesssim \| \bfM^{(1,1)} \|^2 \frac{n}{p^2}\| \bfD^{-1/2} \!\! - \tilde\bfD^{-1/2} \|_F^2 
    + \| \bfD^{-1/2} \!\! - \tilde\bfD^{-1/2}  \|^2 \frac{n}{p^2}\| \tilde\bfM^{(2)}  \|_F^2 
     = o_{\P}(1). 
\end{align*}
Similarly, one can show that 
\begin{align*}
    \frac{n}{p^2} \left\| \lb \bfD^{-1/2} - \tilde\bfD^{-1/2} \rb \bfM \bfD^{-1/2}  \right\|_F^2 
    = o_{\P}(1). 
\end{align*}
Combining these bounds with \eqref{rep_t_m}, we get 
\begin{align}
   &  \frac{n}{p^2} \left\| \bfT  -   \tilde\bfD^{-1/2} \bfM \tilde\bfD^{-1/2}  
    \right\|_F^2  
     = 
     \frac{n}{p^2} \left\| \bfD^{-1/2} \bfM \bfD^{-1/2} - \tilde\bfD^{-1/2} \bfM \tilde\bfD^{-1/2} \right\|_F^2  \nonumber \\
     & \lesssim \frac{n}{p^2} \left\| \lb \bfD^{-1/2} - \tilde\bfD^{-1/2} \rb \bfM \bfD^{-1/2}  \right\|_F^2 
     + \frac{n}{p^2} \left\|    \tilde\bfD^{-1/2}  \bfM \lb \bfD^{-1/2} - \tilde\bfD^{-1/2} \rb \right\|_F^2 
     \nonumber \\ 
    &  = o_{\P}(1). \label{z1}
\end{align}
Thus, LSDs of $\sqrt{n/p}~\,\tilde \bfD^{-1/2} \bfM  \tilde \bfD^{-1/2}$ and $\sqrt{n/p}~\bfT$ coincide, and the proof of claim $(i)$ concludes. 

\subsection{Proof of claim $(ii)$} \label{sec_proof_kendall_claim2}
By \cite[Corollary A.41]{bai:silverstein:2010}, the \Levy distance between the ESDs of $\sqrt{n/p}\,\tilde \bfD^{-1/2} \bfM  \tilde \bfD^{-1/2}$ and $\sqrt{n/p}\,\tilde \bfD^{-1/2} \bfM^{(1,1)}  \tilde \bfD^{-1/2}$ converges to zero (in probability) if $1/p$ times the squared Frobenius norm $\| \cdot \|_F$ of the difference between these two matrices tends to zero (in probability), that is,
\begin{equation}\label{eq:dfgsded}
\frac{1}{p} \frac{n}{p} \Big\|\tilde \bfD^{-1/2} \big(\bfM-\bfM^{(1,1)} \big)  \tilde \bfD^{-1/2}\Big\|_F^2\cip 0\,, \qquad \pto\,.
\end{equation}
By \eqref{tilde_d_ii}, we may assume that $\min_{1\le i\le p} \tilde D_{ii}>C$ for the remainder of this proof.
In view of \eqref{decomp_Ttilde} and the triangle inequality for the Frobenius norm, equation \eqref{eq:dfgsded} holds if we can show that the terms
\begin{align}
\frac{n}{p^2} \left\| \tilde\bfD^{-1/2} \bfM^{(a)} \tilde\bfD^{-1/2} \right\|_F^2, \qquad  a\in \{2,3,4\}
    \label{neg_terms}
\end{align}
 converge to $0$ (in probability). 
For $a=4$, the assertion \eqref{neg_terms} follows from 
\begin{equation*}
\begin{split}
&\frac{n}{p^2}  \E \Big[ \Big\| \frac{2}{n(n-1)} \sum_{1\le s\neq t\le n} \tilde \bfD^{-1/2} \offdiag( \bfu_s \bfu_t') \tilde \bfD^{-1/2} \Big\|_F^2 \Big]
 \lesssim \frac{1}{n^{3} p^2} \sum_{1 \leq i \neq j \leq p} \E \Bigg[\Bigg( \sum_{1\le s\neq t\le n} \frac{ U_{is} U_{jt} }{ \sqrt{\tilde D_{ii} \tilde D_{jj} } } \Bigg)^2\Bigg]\\
&\lesssim \frac{1}{n^{3} p^2} \sum_{1 \leq i \neq j \leq p} \E \Bigg[\Bigg( \sum_{1\le s\neq t\le n}  U_{is} U_{jt}  \Bigg)^2\Bigg]
\lesssim \frac{1}{n^{3} p^2} \sum_{1 \leq i \neq j \leq p} \sum_{1\le s\neq t\le n} \E \Big[   U_{is}^2 U_{jt}^2  \Big]\lesssim \frac{1}{n}\,,
\end{split}
\end{equation*}
since $\E[U_{is}]=0$. For $a=3$, we get similarly
\begin{align*}
&\frac{n}{p^2} \left\| \tilde\bfD^{-1/2} \bfM^{(3)} \tilde\bfD^{-1/2} \right\|_F^2 \lesssim \frac{1}{n^{3} p^2} \sum_{1 \leq i \neq j \leq p} \E \Bigg[\Bigg( \sum_{1\le s\neq t\le n} \frac{ Z_{i,s,t} Z_{j,s,t} }{ \sqrt{\tilde D_{ii} \tilde D_{jj} } } \Bigg)^2\Bigg]\\
&\lesssim \frac{1}{n^{3} p^2} \sum_{1 \leq i \neq j \leq p} \E \Bigg[\Bigg( \sum_{1\le s\neq t\le n}  Z_{i,s,t} Z_{j,s,t}  \Bigg)^2\Bigg]\lesssim \frac{1}{n}\,,
\end{align*}
where $Z_{i,s,t}:= \sign (X_{is}-X_{it})-U_{is}+U_{it}$. For the last step we used that $\E[Z_{i,s,t}Z_{i,s',t'}]=0$ if $\{s,t\}\neq \{s',t'\}$, which follows from the tower property of conditional expectation. 
The argument for $a=2$ is completely analogous, establishing \eqref{neg_terms} and completing the proof of claim $(ii)$.

\subsection{Proof of claim $(iii)$} \label{sec_proof_kendall_claim3}
Assume $\gamma>0$ and note that
\begin{align*}
\tilde \bfD^{-1/2} \bfM^{(1,1)}  \tilde \bfD^{-1/2} &= 
 \frac{2}{n } \sum\limits_{s=1}^n  \tilde \bfD^{-1/2} \bfu_s \bfu_s' \tilde \bfD^{-1/2}
     - \frac{2}{3}  \bfI\\
		&= \frac{2}{3} \Y\Y'
     - \frac{2}{3} \bfI\,,
\end{align*}
where $\Y:=\sqrt{3/n}\, \tilde\bfD^{-1/2} (\bfu_1, \ldots, \bfu_{n})$. The matrix $\Y$ has independent rows on the Euclidean unit sphere and one can check that it 
satisfies the conditions of Theorem~\ref{thm_unit_sphere}.
Therefore, by Theorem~\ref{thm_unit_sphere} part (2), the ESD of $\frac{2}{3} \Y\Y'$ converges a.s.\ to the distribution of $\tfrac{2}{3} \eta$, which establishes the first part of claim $(iii)$.

In the case $\gamma=0$, an application of Theorem~\ref{thm_unit_sphere} part (1)  yields  that the ESD of 
\begin{align*} 
   \sqrt{\frac{n}{p}}  \tilde\bfD^{-1/2} \bfM^{(1,1)} \tilde\bfD^{-1/2}  &= \frac{2}{3} \sqrt{\frac{n}{p}} \big(\Y\Y' - \bfI\big)
\end{align*}
converges a.s.\ to the distribution of $\frac{2}{3}\zeta$, completing the proof of claim $(iii)$.

\section*{Funding}
The work of N.\ Dörnemann  was supported by the  
 DFG Research unit 5381 {\it Mathematical Statistics in the Information Age}, project number 460867398, and the Aarhus University Research Foundation (AUFF), project numbers 47221 and 47388.
 J.\ Heiny’s research was supported by the Swedish Research Council grant VR-2023-03577 ``High-dimensional extremes and random matrix structures'' and by the Verg-Foundation.

 \section*{Acknowledgments}

 We are grateful to the referees, the Associate Editor and Editor-in-Chief for their helpful comments and suggestions on an earlier version of this manuscript.

\bibliographystyle{plainnat}
  \bibliography{libraryAug2022.bib}

\end{document}